\documentclass[11pt,a4paper]{article}
\usepackage[a4paper]{geometry}
\usepackage{amssymb,latexsym,amsmath,amsfonts,amsthm}
\usepackage{graphicx}
\usepackage{algorithm}
\usepackage{algorithmic}
\usepackage{dsfont}
\usepackage{mathrsfs}
\usepackage{epsfig}
\usepackage{booktabs}
\usepackage{mathtools}
\usepackage{overpic}
\usepackage{hyperref}
\usepackage{bbm}
\usepackage[toc,page]{appendix}

\newtheorem{theorem}{Theorem}[section]

\newtheorem{corollary}[theorem]{Corollary}

\theoremstyle{definition}

\theoremstyle{definition}
\newtheorem{example}[theorem]{Example}

\theoremstyle{remark}

\newtheorem{remark}[theorem]{Remark}

\numberwithin{equation}{section}

\usepackage{color}

\usepackage[normalem]{ulem}

\begin{document}
\title{New error bounds for Legendre approximations of differentiable functions}
\author{Haiyong Wang\footnotemark[1]~\footnotemark[2]}
\date{}
\maketitle

\footnotetext[1]{School of Mathematics and Statistics, Huazhong
University of Science and Technology, Wuhan 430074, P. R. China.
E-mail: \texttt{haiyongwang@hust.edu.cn}}

\footnotetext[2]{Hubei Key Laboratory of Engineering Modeling and
Scientific Computing, Huazhong University of Science and Technology,
Wuhan 430074, China.}

\begin{abstract}
In this paper we present a new perspective on error analysis for
Legendre approximations of differentiable functions. We start by
introducing a sequence of Legendre-Gauss-Lobatto polynomials and
prove their theoretical properties, including an explicit and
optimal upper bound. We then apply these properties to derive a new
explicit bound for the Legendre coefficients of differentiable
functions. Building on this, we establish an explicit and optimal
error bound for Legendre approximations in the $L^2$ norm and an
explicit and optimal error bound for Legendre approximations in the
$L^{\infty}$ norm under the condition that their maximum error is
attained in the interior of the interval. Illustrative examples are
provided to demonstrate the sharpness of our new results.
\end{abstract}


\vspace{0.05in}


\section{Introduction}\label{sec:introduction}
Legendre approximations are one of the most fundamental methods in
the field of scientific computing, such as Gauss-type quadrature,
finite element and spectral methods for the numerical solution of
differential equations (see, e.g.,
\cite{canuto2006spectral,ern2021,gautschi2004Orth,shen2011spectral}).
One of the most remarkable advantages of Legendre approximations is
that their accuracy depends solely upon the smoothness of the
underlying functions. From both theoretical and applied
perspectives, it is of particular importance to study error
estimates of various Legendre approximation methods, such as
Legendre projection and interpolation.

Over the past one hundred years, there has been a continuing
interest in developing error estimates and error bounds for
classical spectral approximations (i.e., Chebyshev, Legendre,
Jacobi, Laguerre and Hermite approximations) and many results can be
found in monographes on approximation theories, orthogonal
polynomials and spectral methods (see, e.g.,
\cite{canuto2006spectral,davis1975interp,Jackson1930,Mason2003,shen2011spectral,szego1975orth,trefethen2019}).
However, the existing results have several typical drawbacks: (i)
optimal error estimates of Laguerre and Hermite approximations are
far from satisfactory; (ii) error estimates and error bounds for
Legendre and, more generally, Gegenbauer and Jacobi approximations
of differentiable functions might be suboptimal. Indeed, for the
former, little literature is available on optimal error estimates or
sharp error bounds for Laguerre and Hermite interpolation. For the
latter, even for the very simple function $f(x)=|x|$, the error
bounds of its Legendre approximation of degree $n$ in the maximum
norm in \cite{liu2020legendre,wang2012legendre,wang2018new} are
suboptimal since they all behave like $O(n^{-1/2})$ as
$n\rightarrow\infty$, while the actual convergence rate is
$O(n^{-1})$ (see \cite[Theorem~3]{wang2021legendre}). In view of
these drawbacks, optimal error estimates and sharp error bounds for
classical spectral approximations have received renewed interest in
recent years. We refer to the monograph \cite{trefethen2019} and the
literature
\cite{babuska2019,liu2020legendre,wang2012legendre,wang2016gegenbauer,wang2018new,wang2021legendre,wang2021gegenbauer,wang2021cheby,xiang2020jacobi,xie2013exp}
for more detailed discussions.

Let $\Omega=[-1,1]$ and let $P_k(x)$ be the Legendre polynomial of
degree $k$. It is well known that the sequence
$\{P_k(x)\}_{k=0}^{\infty}$ forms a complete orthogonal system on
$\Omega$ and
\begin{align}\label{def:LegOrth}
\int_{\Omega} P_j(x) P_k(x) \mathrm{d}x = \left(k + \frac{1}{2}
\right)^{-1} \delta_{j,k},
\end{align}
where $\delta_{j,k}$ is the Kronecker delta. For any
$f\in{L}^2(\Omega)$, its Legendre projection of degree $n$ is
defined by
\begin{align}\label{eq:LegExp}
f_n(x) = \sum_{k=0}^{n} a_k P_k(x), \quad  a_k = \left(k +
\frac{1}{2} \right) \int_{-1}^{1} f(x) P_k(x) \mathrm{d}x.
\end{align}
In order to analyze the error estimate of Legendre projections in
both the $L^2$ and $L^{\infty}$ norms, sharp estimates of the
Legendre coefficients play an important role in the analysis (see,
e.g.,
\cite{liu2020legendre,wang2012legendre,wang2016gegenbauer,wang2018new,wang2021legendre,wang2021cheby,xiang2020jacobi}).
Indeed, these estimates are useful not only in understanding the
convergence rates of Legendre projections but useful also in
estimating the degree of the Legendre projection to approximate
$f(x)$ within a prescribed accuracy. In recent years, sharp
estimates of the Legendre coefficients have experienced rapid
development. In the case when $f(x)$ is analytic inside and on the
Bernstein ellipse
\begin{align}
\mathcal{E}_{\rho} = \left\{z\in \mathbb{C} ~\bigg| ~z = \frac{u +
u^{-1}}{2}, ~ u=\rho e^{\mathrm{i}\theta},~ 0\leq\theta<2\pi
\right\},
\end{align}
for some $\rho>1$, an explicit and sharp bound for the Legendre
coefficients was given in \cite[Lemma~2]{wang2021legendre}
\begin{align}\label{eq:akBound}
|a_0| \leq  \frac{D(\rho)}{2}, \qquad |a_k|\leq
D(\rho)\frac{\sqrt{k}}{\rho^k}, \quad k\geq1,
\end{align}
where $D(\rho)=2M(\rho)L(\mathcal{E}_{\rho})/(\pi\sqrt{\rho^2-1})$
and $M(\rho)=\max_{z\in\mathcal{E}_{\rho}} |f(z)|$ and
$L(\mathcal{E}_{\rho})$ is the length of the circumference of
$\mathcal{E}_{\rho}$. As a direct consequence, for each $n\geq0$, it
was proved in \cite[Theorem~2]{wang2021legendre} that
\begin{align}\label{eq:LegMaxError}
\|f-f_n\|_{L^{\infty}(\Omega)} &\leq \frac{D(\rho)}{\rho^n} \bigg(
\frac{(n+1)^{1/2}}{\rho-1} + \frac{(n+1)^{-1/2}}{(\rho-1)^2} \bigg).
\end{align}
Moreover, another direct consequence of \eqref{eq:akBound} is the
error bound of Legendre projection in the $L^2$ norm:
\begin{align}\label{eq:LegMean}
\|f-f_n\|_{L^{2}(\Omega)} &\leq %\left[ \sum_{k=n+1}^{\infty} (a_k)^2
%\left(k + \frac{1}{2}\right)^{-1} \right]^{1/2} \leq
\left( \sum_{k=n+1}^{\infty} \frac{D(\rho)^2}{\rho^{2k}}
\right)^{1/2} = \frac{D(\rho)}{\rho^n\sqrt{\rho^2-1}}.
\end{align}
In the case when $f(x)$ is differentiable but not analytic on the
interval $\Omega$, upper bounds for the Legendre coefficients were
extensively studied in
\cite{liu2020legendre,wang2012legendre,wang2018new,xiang2020jacobi}.
However, those bounds in \cite{wang2012legendre,wang2018new} depend
on some semi-norms of high-order derivatives of $f(x)$ which may be % largely
overestimated, especially when the singularity is close to
endpoints, and those bounds in
\cite{liu2020legendre,xiang2020jacobi} involve ratios of gamma
functions, which are less favorable since their asymptotic behavior
and computation still require further treatment.

In this paper, we give a new perspective on error bounds of Legendre
approximations for differentiable functions. We start by introducing
the Legendre-Gauss-Lobatto (LGL) polynomials
\begin{align}\label{def:PhiFun}
\phi_k^{\mathrm{LGL}}(x) = \left\{
\begin{array}{ll}
P_{k+1}(x) - P_{k-1}(x), & \hbox{$k\geq1$,}   \\[6pt]
P_1(x),           & \hbox{$k=0$,}
            \end{array}
            \right.
\end{align}
and then proving theoretical properties of these polynomials,
including the differential recurrence relation and an optimal and
explicit bound for their maximum value on $\Omega$. Based on LGL
polynomials, we obtain a new explicit and sharp bound for the
Legendre coefficients of differentiable functions, which is sharper
than the result in \cite{wang2018new} and is more informative than
the results in \cite{liu2020legendre,xiang2020jacobi}. Building on
this new bound, we then establish an explicit and optimal error
bounds for Legendre projections in the $L^2$ norm and an explicit
and optimal error bound for Legendre projections in the $L^\infty$
norm whenever the maximum error of $f_n$ is attained in the interior
of $\Omega$. We emphasize that in contrast to those results in
\cite{liu2020legendre,xiang2020jacobi} which involve ratios of gamma
functions, our results are more explicit and informative.

This paper is organized as follows. In section \ref{sec:properties}
we prove some theoretical properties of the LGL polynomials. In
section \ref{sec:coefficient} we establish a new bound for the
Legendre coefficients of differentiable functions, which improves
the existing result in \cite{wang2018new}. Building on this bound,
we establish some new error bounds of Legendre projections in both
$L^2$ and $L^{\infty}$ norms. In section \ref{sec:extension} we
present two extensions, including the extension of LGL polynomials
to Gegenbauer-Gauss-Lobatto functions and optimal convergence rates
of LGL interpolation and differentiation for analytic functions.
Finally, we give some concluding remarks in section
\ref{sec:Conclusion}.

%--------------------------------------------------------------------------------------
\section{Properties of Legendre-Gauss-Lobatto polynomials}\label{sec:properties}
In this section, we establish some theoretical properties of LGL
polynomials $\{\phi_k^{\mathrm{LGL}}\}$. Our main result is stated
in the following theorem.
\begin{theorem}\label{thm:PhiFun}
Let $\phi_k^{\mathrm{LGL}}(x)$ be defined in \eqref{def:PhiFun}.
Then, the following properties hold:
\begin{itemize}
\item[\rm (i)] $|\phi_k^{\mathrm{LGL}}(x)|$ is even for $k=0,1,\ldots$ and $\phi_k^{\mathrm{LGL}}(\pm1)=0$ for $k=1,2,\ldots$.

\item[\rm (ii)] The following differential recurrence relation holds
\begin{align}\label{eq:LegDiffRec}
\phi_k^{\mathrm{LGL}}(x) = \frac{\mathrm{d}}{\mathrm{d}x} \left(
\frac{\phi_{k+1}^{\mathrm{LGL}}(x)}{2k+3} -
\frac{\phi_{k-1}^{\mathrm{LGL}}(x)}{2k-1} \right), \quad
k=1,2,\ldots.
\end{align}

\item[\rm (iii)] Let $\nu=\lfloor (n+1)/2 \rfloor$ and let
$x_{\nu}<\cdots<x_1$ be the zeros of $P_n(x)$ on the interval
$[0,1]$. Then, $|\phi_n^{\mathrm{LGL}}(x)|$ attains its local
maximum values at these points and
\begin{align}\label{eq:DecSeq}
|\phi_n^{\mathrm{LGL}}(x_1)|<|\phi_n^{\mathrm{LGL}}(x_2)|<\cdots<|\phi_n^{\mathrm{LGL}}(x_{\nu})|.
\end{align}

\item[\rm (iv)] The maximum value of $|\phi_{n}^{\mathrm{LGL}}(x)|$ satisfies
\begin{align}\label{eq:Ineq}
\max_{x\in\Omega}|\phi_{n}^{\mathrm{LGL}}(x)| \leq
\frac{4}{\sqrt{2\pi n}}, \quad n=1,2,\ldots,
\end{align}
and the bound on the right-hand side is optimal in the sense that it
can not be improved further.
\end{itemize}
\end{theorem}
\begin{proof}
As for (i), they follow from the properties of Legendre polynomials,
i.e., $P_k(-x)=(-1)^kP_k(x)$ and $P_k(\pm1)=(\pm1)^k$ for
$k=0,1,\ldots$. As for (ii), we obtain from
\cite[Equation~(4.7.29)]{szego1975orth} that
\begin{align}\label{eq:LegRec}
\frac{\mathrm{d}}{\mathrm{d}x} \phi_k^{\mathrm{LGL}}(x) = (2k+1)
P_k(x), \quad k=0,1,\ldots.
\end{align}
The identity \eqref{eq:LegDiffRec} follows immediately from
\eqref{eq:LegRec}. As for (iii), due to \eqref{eq:LegRec} we know
that $|\phi_n^{\mathrm{LGL}}(x)|$ attains its local maximum values
at the zeros of $P_n(x)$. Since $|\phi_n^{\mathrm{LGL}}(x)|$ is an
even function on $\Omega$, we only need to consider the maximum
values of $|\phi_n^{\mathrm{LGL}}(x)|$ at the nonnegative zeros of
$P_n(x)$. To this end, we introduce an auxiliary function
\begin{align}\label{eq:AuxFun}
\psi(x) = \frac{n(n+1)}{(2n+1)^2}(\phi_n^{\mathrm{LGL}}(x))^2 +
(1-x^2)P_n^2(x).
\end{align}
Direct calculation of the derivative of $\psi(x)$ by using
\eqref{eq:LegRec} gives
\begin{align}\label{eq:AuxFunD}
\psi{'}(x) &= \frac{2n(n+1)}{2n+1}\phi_n^{\mathrm{LGL}}(x)P_n(x) -
2xP_n^2(x) + 2(1-x^2)P_n(x)P_n{'}(x) \nonumber \\
&= 2P_n(x) \left[ \frac{n(n+1)}{2n+1}\phi_n^{\mathrm{LGL}}(x) -
xP_n(x) + (1-x^2)P_n{'}(x) \right] \nonumber \\
&= -2x P_n^2(x),
\end{align}
where we have used the first equality of
\cite[Equation~(4.7.27)]{szego1975orth} in the last step. From
\eqref{eq:AuxFunD} it is easy to see that $\psi{'}(x)<0$ whenever
$x\in[0,1]$ and $x\neq x_j$ for $j=1,\ldots,\mu$, and thus $\psi(x)$
is strictly decreasing on the interval $[0,1]$. Combining this with
\eqref{eq:AuxFun} gives the desired result \eqref{eq:DecSeq}. As for
(iv), it follows from \eqref{eq:DecSeq} that
\begin{align}\label{eq:MaxVal}
\max_{x\in\Omega}|\phi_{n}^{\mathrm{LGL}}(x)| =
|\phi_n^{\mathrm{LGL}}(x_{\nu})|.
\end{align}
Now, we first consider the case where $n$ is odd. In this case, we
know that $x_{\nu}=0$ and thus
\begin{align}
\max_{x\in\Omega}|\phi_{n}^{\mathrm{LGL}}(x)|=|\phi_n^{\mathrm{LGL}}(0)|
=
\frac{(2n+1)\Gamma(\frac{n}{2})}{\sqrt{\pi}(n+1)\Gamma(\frac{n+1}{2})}
:= \chi_n.  \nonumber
\end{align}
A straightforward calculation shows that the sequence $\{
n^{1/2}\chi_n\}_{n=1}^{\infty}$ is a strictly increasing sequence,
and thus
\begin{align}
n^{1/2}\chi_n \leq \lim_{n\rightarrow\infty} n^{1/2} \chi_n =
\frac{4}{\sqrt{2\pi}} \quad \Longrightarrow \quad
\max_{x\in\Omega}|\phi_{n}^{\mathrm{LGL}}(x)| \leq
\frac{4}{\sqrt{2\pi n}}. \nonumber
\end{align}
This proves the case of odd $n$. We next consider the case where $n$
is even. Recall that $\psi(x)$ is strictly decreasing on the
interval $[0,1]$, we obtain that
\begin{align}
\psi(x) \leq \psi(0) \quad \Longrightarrow \quad
\max_{x\in\Omega}|\phi_n^{\mathrm{LGL}}(x)| \leq
\frac{(2n+1)\Gamma(\frac{n+1}{2})}{\sqrt{\pi
n(n+1)}\Gamma(\frac{n+2}{2})} := \zeta_n.  \nonumber
\end{align}
A straightforward calculation shows that the sequence
$\{n^{1/2}\zeta_n\}_{n=1}^{\infty}$ is a strictly increasing
sequence, and thus
\begin{align}
n^{1/2} \zeta_n \leq \lim_{n\rightarrow\infty} n^{1/2} \zeta_n =
\frac{4}{\sqrt{2\pi}} \quad \Longrightarrow \quad
\max_{x\in\Omega}|\phi_n^{\mathrm{LGL}}(x)| \leq \frac{4}{\sqrt{2\pi
n}}. \nonumber
\end{align}
This proves the case of even $n$. Since the constant $4/\sqrt{2\pi}$
is the limit of the sequences $\{ n^{1/2}\chi_n\}_{n=1}^{\infty}$
and $\{n^{1/2}\zeta_n\}_{n=1}^{\infty}$, the bound in
\eqref{eq:Ineq} is optimal in the sense that it can not be reduced
further. This completes the proof.
\end{proof}

\begin{figure}[h]
\centering
\begin{overpic}
[width=4.8cm,height=5.0cm]{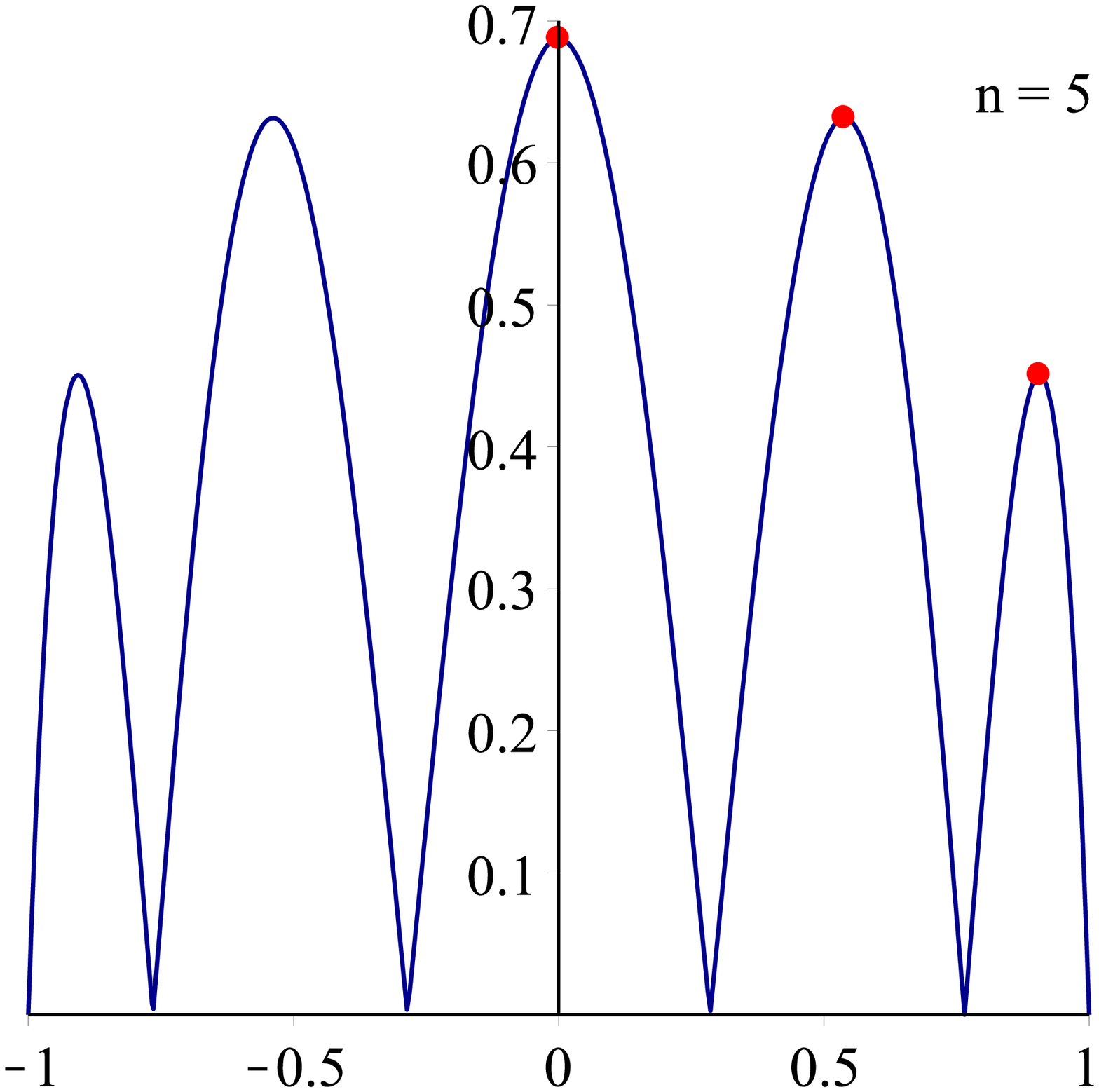}
\end{overpic}
\begin{overpic}
[width=4.8cm,height=5.0cm]{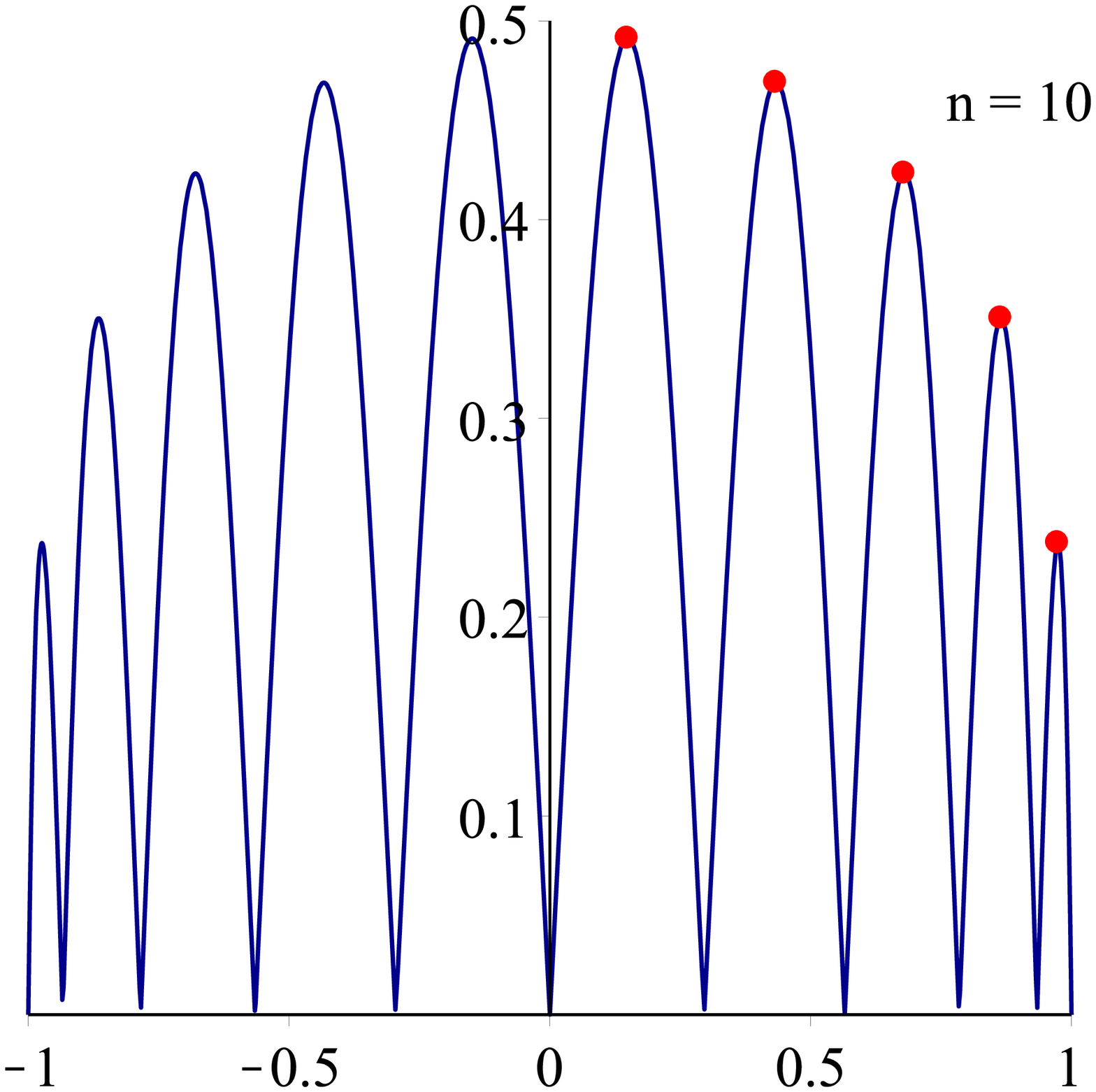}
\end{overpic}
\begin{overpic}
[width=4.8cm,height=5.0cm]{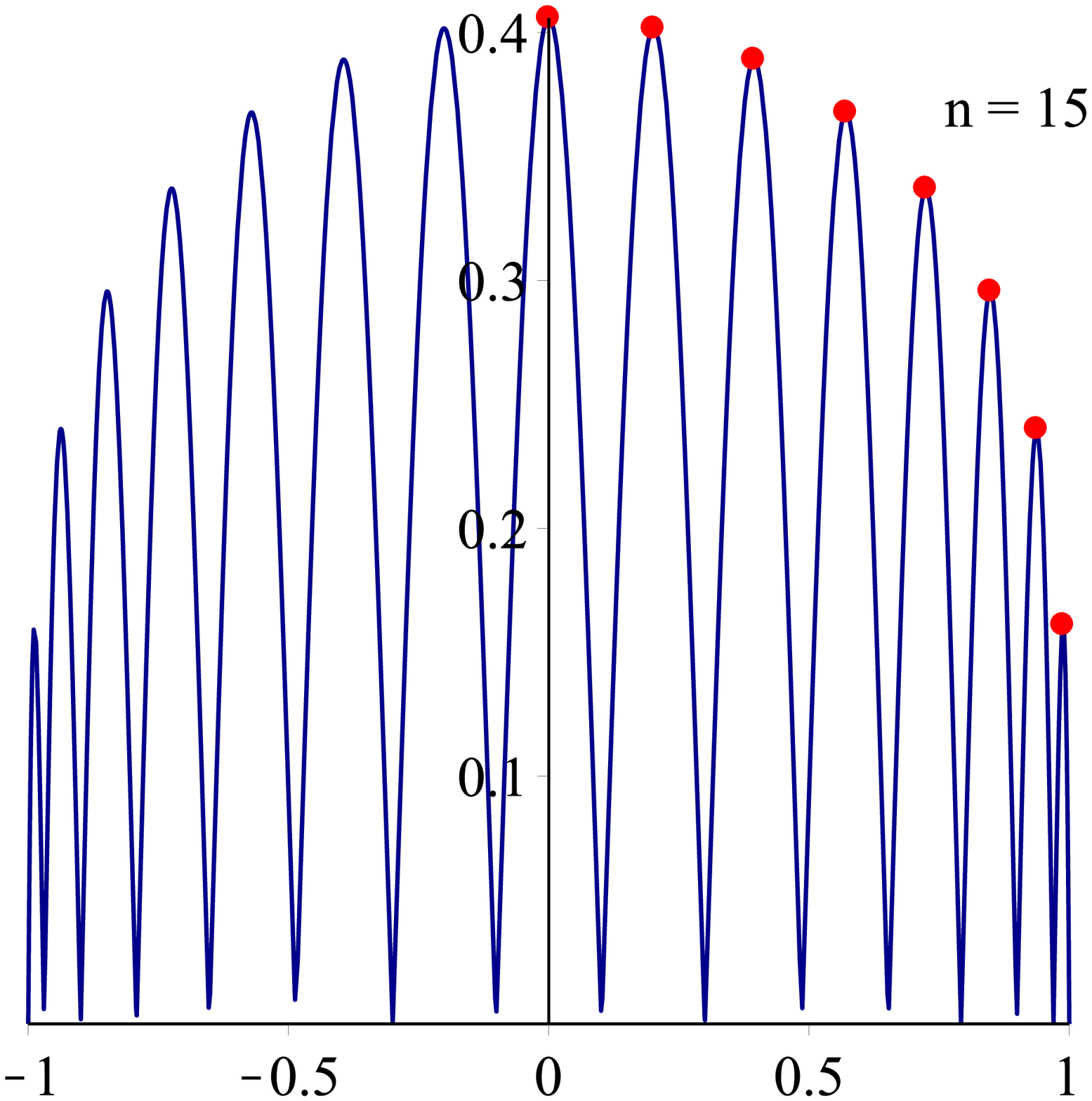}
\end{overpic}
\caption{The plot of $|\phi_n^{\mathrm{LGL}}(x)|$ and the points
$\{(x_j,|\phi_n^{\mathrm{LGL}}(x_j)|)\}_{j=1}^{\nu}$ for three
values of $n$.} \label{fig:LegPolyI}
\end{figure}

In Figure \ref{fig:LegPolyI} we plot the function
$|\phi_n^{\mathrm{LGL}}(x)|$ and the points
$\{(x_j,|\phi_n^{\mathrm{LGL}}(x_j)|)\}_{j=1}^{\nu}$ for
$n=5,10,15$. Indeed, it is easily seen that the sequence
$\{|\phi_n^{\mathrm{LGL}}(x_1)|,\ldots,|\phi_n^{\mathrm{LGL}}(x_{\nu})|\}$
is strictly increasing, which coincides with the result (iii) in
Theorem \ref{thm:PhiFun}. To verify the sharpness of the result (iv)
in Theorem \ref{thm:PhiFun}, we plot in Figure \ref{fig:LegPolyII}
the scaled function
$\phi_n^{S}(x)=|\phi_{n}^{\mathrm{LGL}}(x)|\sqrt{2\pi n}/4$ for
$x\in\Omega$. Clearly, we observe that the maximum value of
$|\phi_n^{S}(x)|$ approaches to one as $n\rightarrow\infty$, which
implies the inequality \eqref{eq:Ineq} is quite sharp.
\begin{figure}[h]
\centering
\begin{overpic}
[width=4.8cm,height=5.cm]{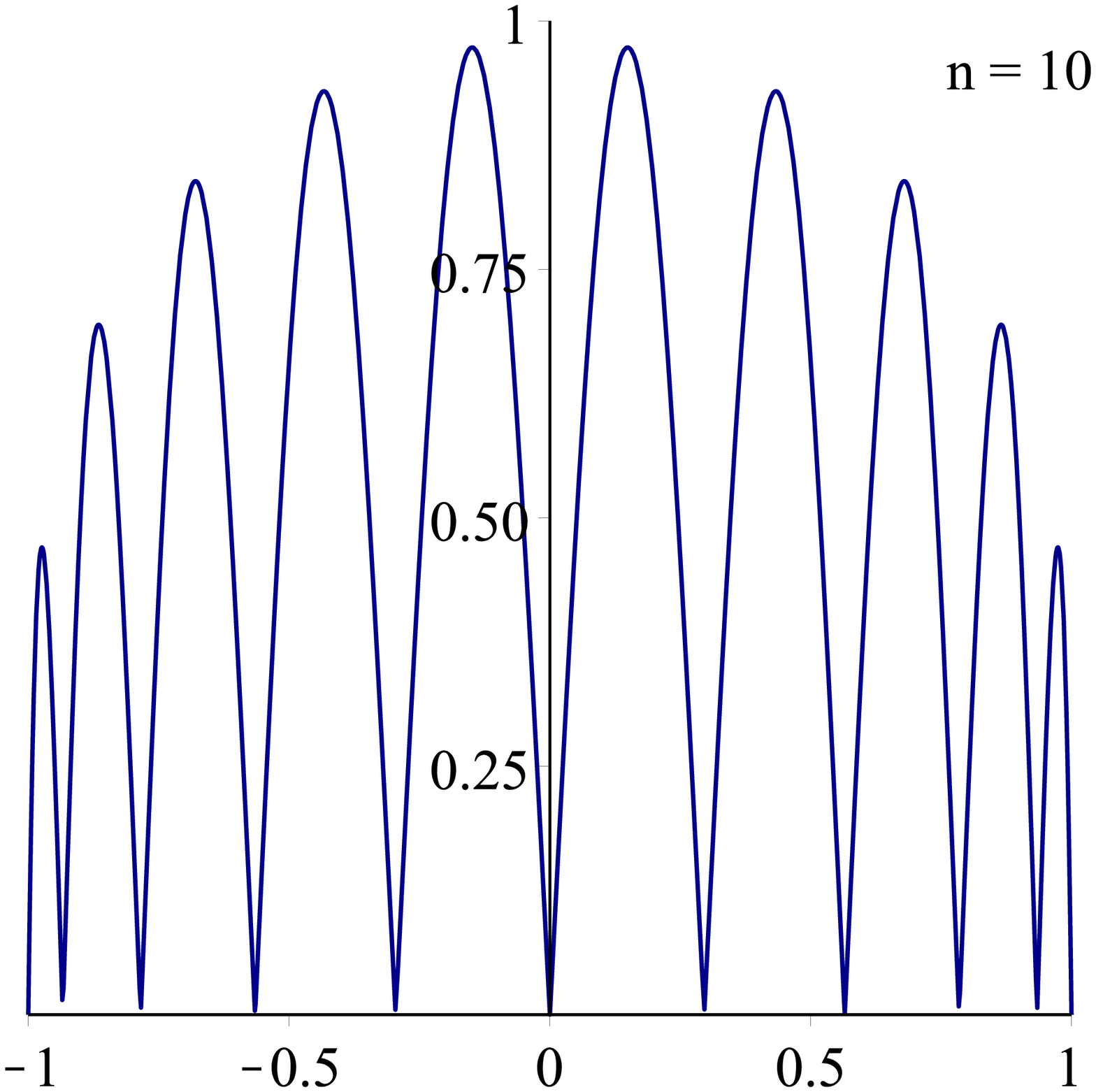}
\end{overpic}
\begin{overpic}
[width=4.8cm,height=5.cm]{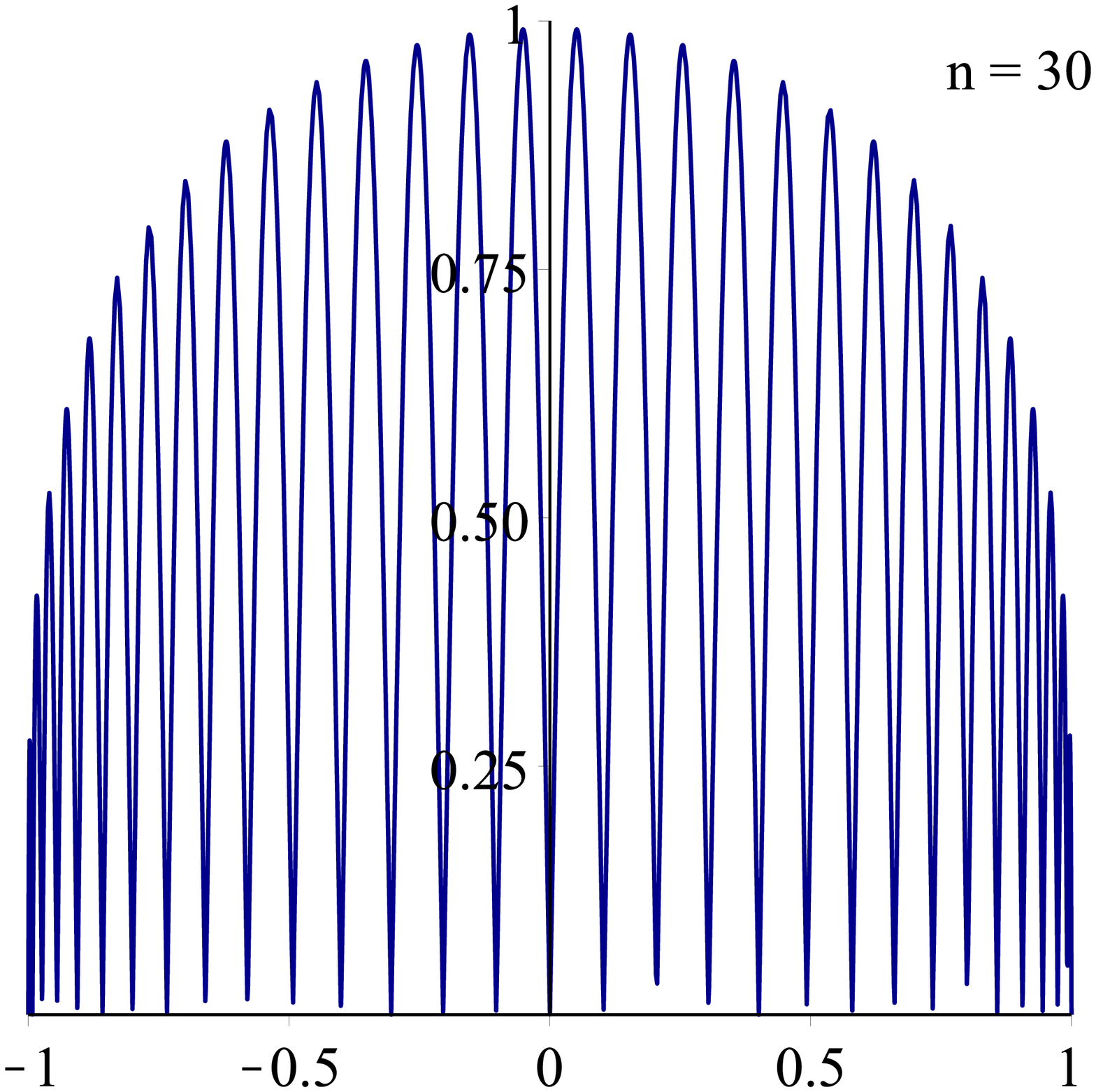}
\end{overpic}
\begin{overpic}
[width=4.8cm,height=5.cm]{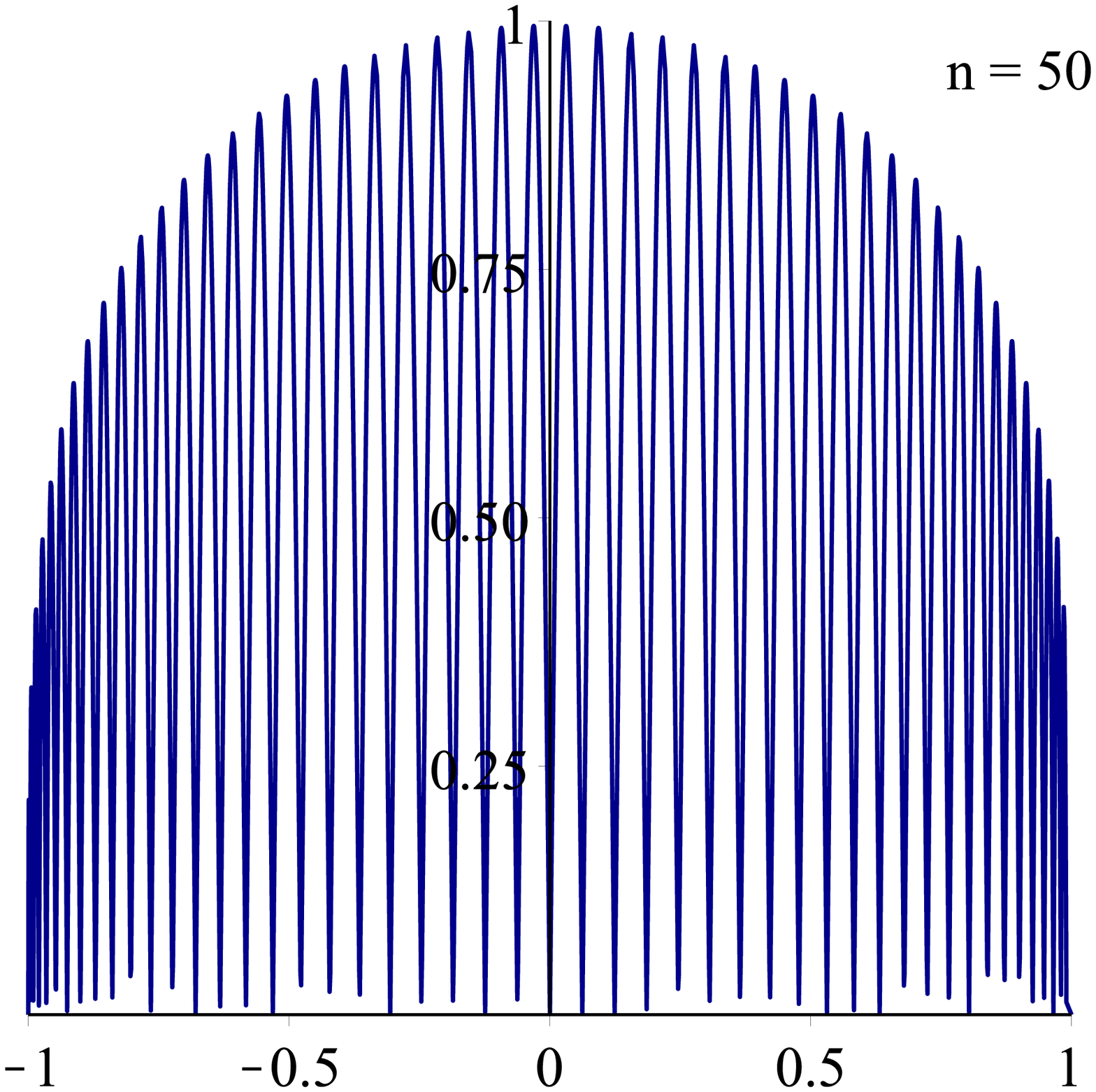}
\end{overpic}
\caption{The plot of the scaled function $\phi_n^{S}(x)$ for three
values of $n$.} \label{fig:LegPolyII}
\end{figure}

\begin{remark}
In the proof of Theorem \ref{thm:PhiFun}, we have actually proved
the following sharper inequality
\begin{align}\label{eq:LGLMax}
\max_{x\in\Omega}|\phi_{n}^{\mathrm{LGL}}(x)| \leq
\frac{(2n+1)}{\sqrt{\pi}} \left\{
\begin{array}{ll}
{\displaystyle \frac{\Gamma(\frac{n}{2})}{(n+1)\Gamma(\frac{n+1}{2})}},  & \hbox{$n=1,3,\ldots$,}   \\[12pt]
{\displaystyle
\frac{\Gamma(\frac{n+1}{2})}{\sqrt{n(n+1)}\Gamma(\frac{n+2}{2})}}, &
\hbox{$n=2,4,\ldots$,}
            \end{array}
            \right.
\end{align}
Since the bound on the right-hand side of \eqref{eq:LGLMax} involves
the ratio of gamma functions, we therefore established a simpler
upper bound for $\max_{x\in\Omega}|\phi_{n}^{\mathrm{LGL}}(x)|$ in
\eqref{eq:Ineq}. On the other hand, we mention that a rough estimate
for $\phi_n^{\mathrm{LGL}}(x)$ has been given in the classical
monograph \cite[Theorem~7.33.3]{szego1975orth}:
\begin{align}\label{eq:SzegoPhi}
\phi_n^{\mathrm{LGL}}(\cos\theta) = (\sin\theta)^{1/2} O(n^{-1/2}),
\quad  0<\theta<\pi,
\end{align}
and the bound of the factor $O(n^{-1/2})$ is independent of
$\theta$. Comparing \eqref{eq:SzegoPhi} with \eqref{eq:Ineq}, it is
clear to see that the latter is more precise than the former.
\end{remark}

\begin{remark}
The polynomials $\{\phi_k^{\mathrm{LGL}}\}$ can also be viewed as
weighted Gegenbauer or Sobolev orthogonal polynomials. Indeed, from
\eqref{def:GGLFun2} in Section \ref{sec:extension} we know that each
$\phi_k^{\mathrm{LGL}}(x)$, where $k\geq1$, can be expressed by a
weighted Gegenbauer polynomial up to a constant factor. On the other
hand, it is easily verified that $\{\phi_k^{\mathrm{LGL}}\}$ are
orthogonal with respect to the Sobolev inner product of the form
\[
\langle f,g \rangle = [f(1)+f(-1)][g(-1)+g(1)] + \int_{\Omega}
f{'}(x) g{'}(x)\mathrm{d}x,
\]
and hence they can also be viewed as Sobolev orthogonal polynomials.
\end{remark}

\section{A new explicit bound for Legendre coefficients of differentiable functions}\label{sec:coefficient}
In this section, we establish a new explicit bound for the Legendre
coefficients of differentiable functions with the help of the
properties of LGL polynomials. As will be shown later, our new
result is better than the existing result in \cite{wang2018new} and
is more informative than the result in \cite{xiang2020jacobi}.

Let the total variation of $f(x)$ on $\Omega$ be defined by
\begin{align}
\mathcal{V}(f,\Omega) = \sup \sum_{j=1}^{n} \left| f(x_j)-f(x_{j-1})
\right|,
\end{align}
where the supremum is taken over all partitions
$-1=x_0<x_1<\cdots<x_n=1$ of the interval $\Omega$.
\begin{theorem}\label{thm:NewBound}
If $f,f{'},\ldots,f^{(m-1)}$ are absolutely continuous on $\Omega$
and $f^{(m)}$ is of bounded variation on $\Omega$ for some
nonnegative integer $m$, i.e.,
$V_m=\mathcal{V}(f^{(m)},\Omega)<\infty$. Then, for each $n\geq
m+1$,
\begin{align}\label{eq:NewBound}
|a_n| &\leq \frac{2V_m}{\sqrt{2\pi(n-m)}} \prod_{k=1}^{m}
\frac{1}{n-k+1/2}.
\end{align}
where the product is assumed to be one whenever $m=0$.
\end{theorem}
\begin{proof}
We prove the inequality \eqref{eq:NewBound} by induction on $m$.
Invoking \eqref{eq:LegRec} and using integration by part, we obtain
\begin{align}
a_n &= \frac{2n+1}{2}\int_{-1}^{1} f(x) P_n(x) \mathrm{d}x =
\frac{1}{2} \int_{-1}^{1} f(x) \frac{\mathrm{d}}{\mathrm{d}x}
\phi_n^{\mathrm{LGL}}(x) \mathrm{d}x = - \frac{1}{2}\int_{-1}^{1}
\phi_n^{\mathrm{LGL}}(x) \mathrm{d}f(x), \nonumber
\end{align}
where we have used Theorem \ref{thm:PhiFun} and the integral is a
Riemann-Stieltjes integral (see, e.g.,
\cite[Chapter~12]{protter1991real}) in the last equation.
Furthermore, taking advantage of the inequality of Riemann-Stieltjes
integral (see, e.g., \cite[Theorem~12.15]{protter1991real}) and
making use of Theorem \ref{thm:PhiFun}, we have
\begin{align}
|a_n| &\leq \frac{1}{2}\left|\int_{-1}^{1}\phi_n^{\mathrm{LGL}}(x)
\mathrm{d}f(x) \right| \leq \frac{V_0}{2}
\max_{x\in\Omega}|\phi_n^{\mathrm{LGL}}(x)| \leq
\frac{2V_0}{\sqrt{2\pi n}}. \nonumber
\end{align}
This proves the case $m=0$. In the case of $m=1$, making use of
\eqref{eq:LegDiffRec} and employing once again integration by part,
we obtain
\begin{align}
a_n &= - \frac{1}{2} \int_{-1}^{1} f{'}(x)
\frac{\mathrm{d}}{\mathrm{d}x}\bigg[
\frac{\phi_{n+1}^{\mathrm{LGL}}(x)}{2n+3} -
\frac{\phi_{n-1}^{\mathrm{LGL}}(x)}{2n-1} \bigg] \mathrm{d}x  \nonumber \\
&= -\frac{1}{2} \int_{-1}^{1} \bigg[
\frac{\phi_{n-1}^{\mathrm{LGL}}(x)}{2n-1} -
\frac{\phi_{n+1}^{\mathrm{LGL}}(x)}{2n+3} \bigg] \mathrm{d}f{'}(x),
\nonumber
\end{align}
from which, using again the inequality of Riemann-Stieltjes integral
and \eqref{eq:Ineq}, we find that
\begin{align}
|a_n| &\leq
\frac{V_1}{2n-1}\max_{x\in\Omega}|\phi_{n-1}^{\mathrm{LGL}}(x)| \leq
\frac{2V_1}{\sqrt{2\pi(n-1)}(n-1/2)}. \nonumber
\end{align}
This proves the case $m=1$. In the case of $m=2$, we can continue
the above process to obtain
\begin{align}
a_n = -\frac{1}{2} \int_{-1}^{1} & \bigg[
\frac{\phi_{n-2}^{\mathrm{LGL}}(x)}{(2n-3)(2n-1)} -
\frac{\phi_{n}^{\mathrm{LGL}}(x)}{(2n-1)(2n+1)} \nonumber \\
&~  - \frac{\phi_{n}^{\mathrm{LGL}}(x)}{(2n+1)(2n+3)} +
\frac{\phi_{n+2}^{\mathrm{LGL}}(x)}{(2n+3)(2n+5)} \bigg]
\mathrm{d}f{''}(x), \nonumber
\end{align}
from which we infer that
\begin{align}
|a_n| \leq \frac{2V_2}{(2n-3)(2n-1)}
\max_{x\in\Omega}|\phi_{n-2}^{\mathrm{LGL}}(x)| \leq
\frac{2V_2}{\sqrt{2\pi(n-2)}(n-1/2)(n-3/2)}. \nonumber
\end{align}
This proves the case $m=2$. For $m\geq3$, the repeated application
of integration by parts brings in higher derivatives of $f$ and
corresponding higher variations up to $V_m$. Hence we can obtain the
desired result \eqref{eq:NewBound} and this ends the proof.
\end{proof}

\begin{remark}
An explicit bound for the Legendre coefficients has been established
in \cite[Theorem~2.2]{wang2018new}
\begin{align}\label{eq:OldBound}
|a_n| \leq \frac{2\overline{V}_m}{\sqrt{\pi(2n-2m-1)}}
\prod_{k=1}^{m} \frac{1}{n-k+1/2},
\end{align}
where $\overline{V}_m=\|f^{(m)}\|_{S}$ and $\|\cdot\|_{S}$ is the
weighted semi-norm defined by $\|f\|_{S}=\int_{\Omega}
(1-x^2)^{-1/4}|f{'}(x)|\mathrm{d}x$. Comparing \eqref{eq:NewBound}
and \eqref{eq:OldBound}, it is easily seen that the weighted
semi-norm of $f^{(m)}(x)$ is replaced by the total variation of
$f^{(m)}(x)$, and \eqref{eq:NewBound} is better than
\eqref{eq:OldBound} because $V_m\leq\overline{V}_m$. Moreover, the
following bound was given in \cite[Corollary~1]{xiang2020jacobi}
\begin{align}\label{eq:Xiang}
|a_n| \leq \frac{V_m}{2^{m}\sqrt{\pi}} \frac{(n+1/2)
\Gamma((n-m)/2)}{(n+m+1) \Gamma((n+m+1)/2)}, \quad  n\geq m+1.
\end{align}
Comparing \eqref{eq:NewBound} with \eqref{eq:Xiang}, one can easily
check that both results are about equally accurate in the sense that
their ratio tends to one as $n\rightarrow\infty$. However, our
result \eqref{eq:NewBound} is more explicit and informative.
\end{remark}

\begin{example}\label{exam:ABS}
We consider the following example
\begin{align}\label{eq:ModFun}
f(x)=|x-\theta|, \quad \theta\in(-1,1).
\end{align}
It is clear that this function is absolutely continuous on $\Omega$
and its derivative is of bounded variation on $\Omega$. Moreover,
direct calculation gives $V_1=2$, and thus
\begin{align}\label{eq:Bnew}
|a_n| \leq \frac{4}{\sqrt{2\pi(n-1)}(n-1/2)} :=
\mathrm{B}^{\mathrm{New}}(n), \quad n=2,3,\ldots.  \nonumber
\end{align}
The upper bound in \eqref{eq:OldBound} can be written as
\begin{align}
|a_n| \leq \frac{4(1-\theta^2)^{-1/4}}{\sqrt{\pi(2n-3)} (n-1/2)} :=
\mathrm{B}^{\mathrm{Old}}(n), \quad n=2,3,\ldots.  \nonumber
\end{align}
Comparing $\mathrm{B}^{\mathrm{New}}(n)$ with
$\mathrm{B}^{\mathrm{Old}}(n)$, it is easily verified that the
former is always better than the latter for all $n\geq2$ and
$\theta\in(-1,1)$, especially the former remains fixed while the
latter blows up when $\theta\rightarrow\pm1$. Figure
\ref{fig:LegExam1} shows the exact Legendre coefficients and the
bound $\mathrm{B}^{\mathrm{New}}(n)$ for three values of $\theta$.
It can be seen that $\mathrm{B}^{\mathrm{New}}(n)$ is quite sharp
whenever $\theta$ is not close to both endpoints and is slightly
overestimated whenever $\theta\rightarrow\pm1$.

\begin{figure}[h]
\centering
\begin{overpic}
[width=4.8cm,height=5.cm]{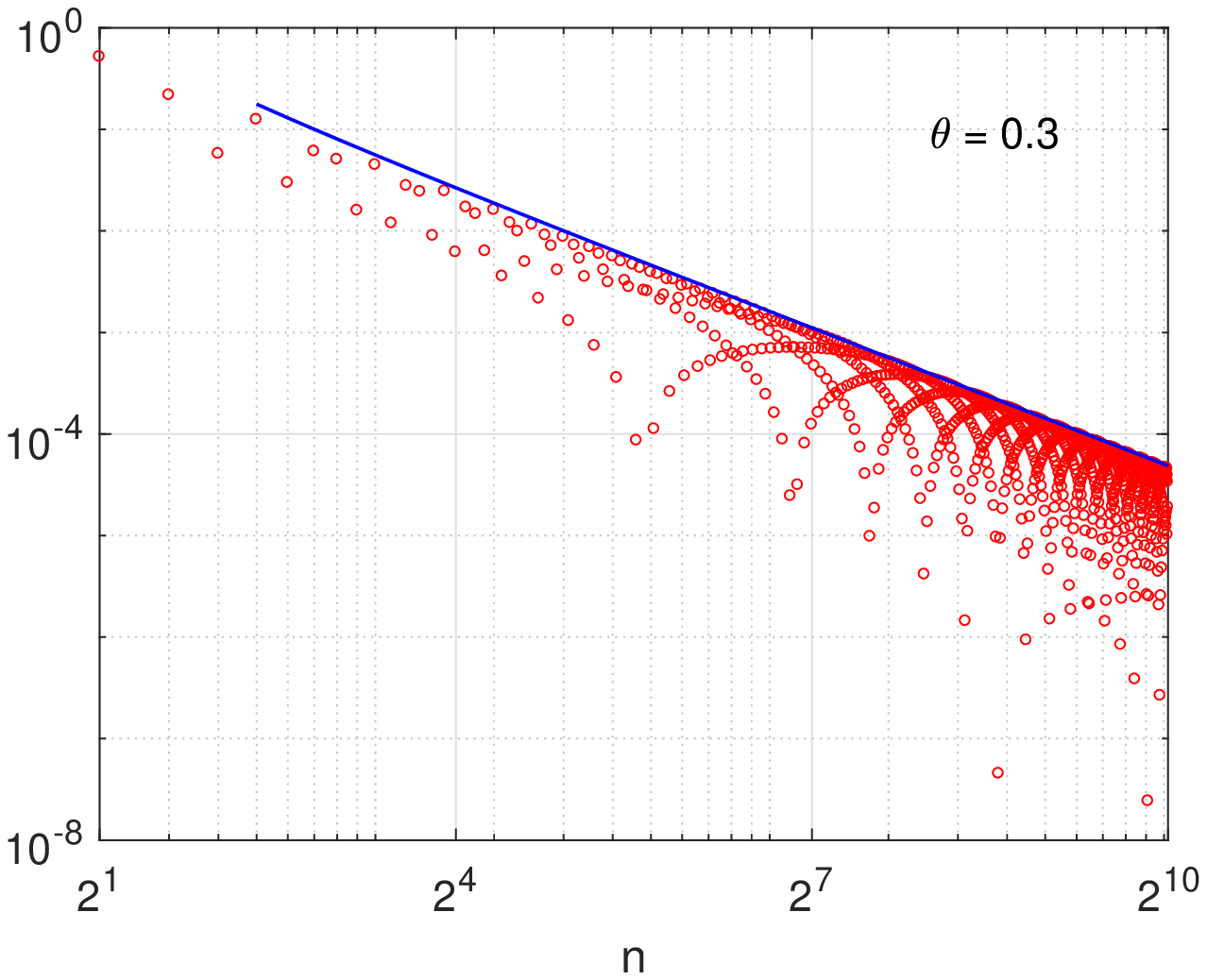}
\end{overpic}
\begin{overpic}
[width=4.8cm,height=5.cm]{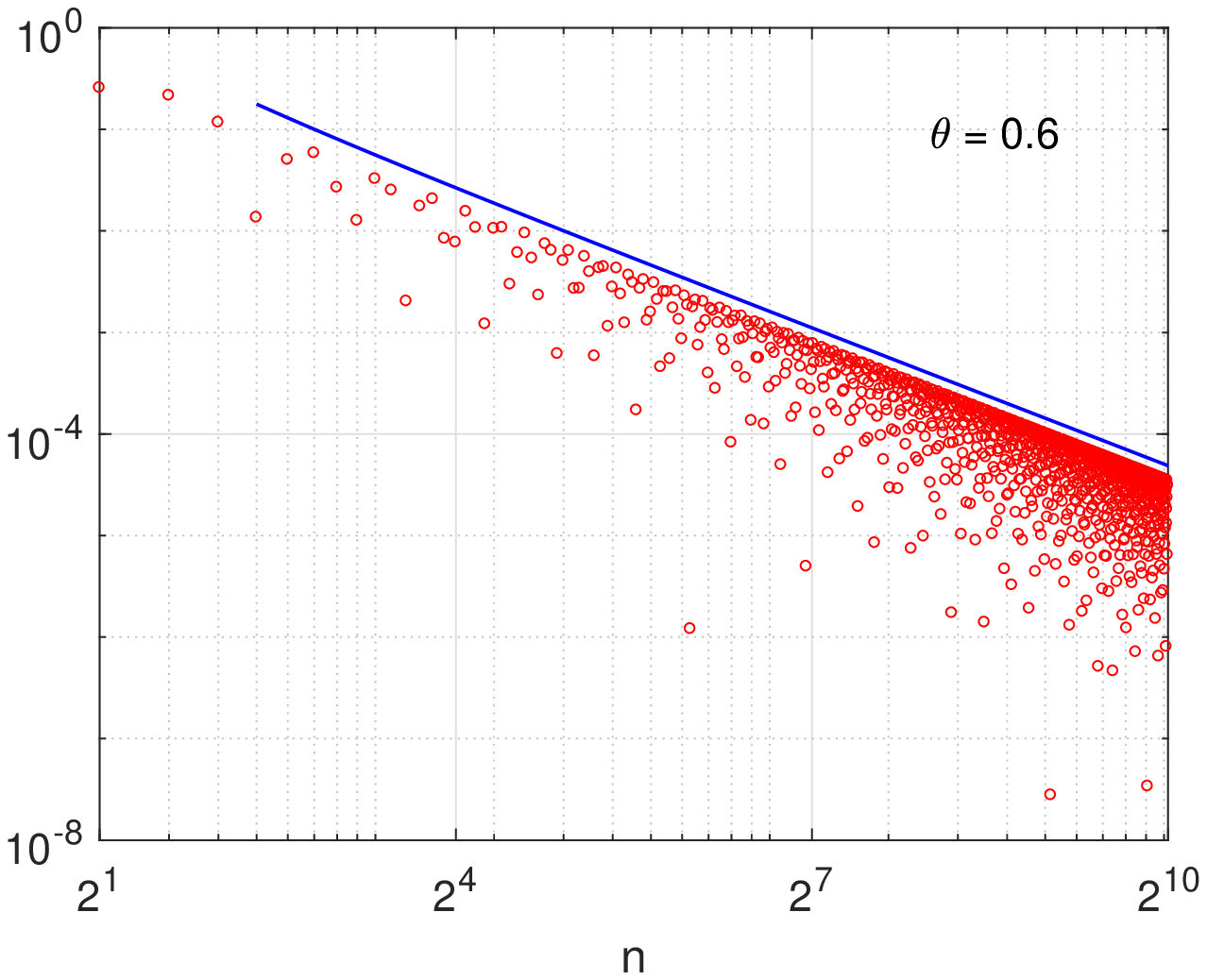}
\end{overpic}
\begin{overpic}
[width=4.8cm,height=5.cm]{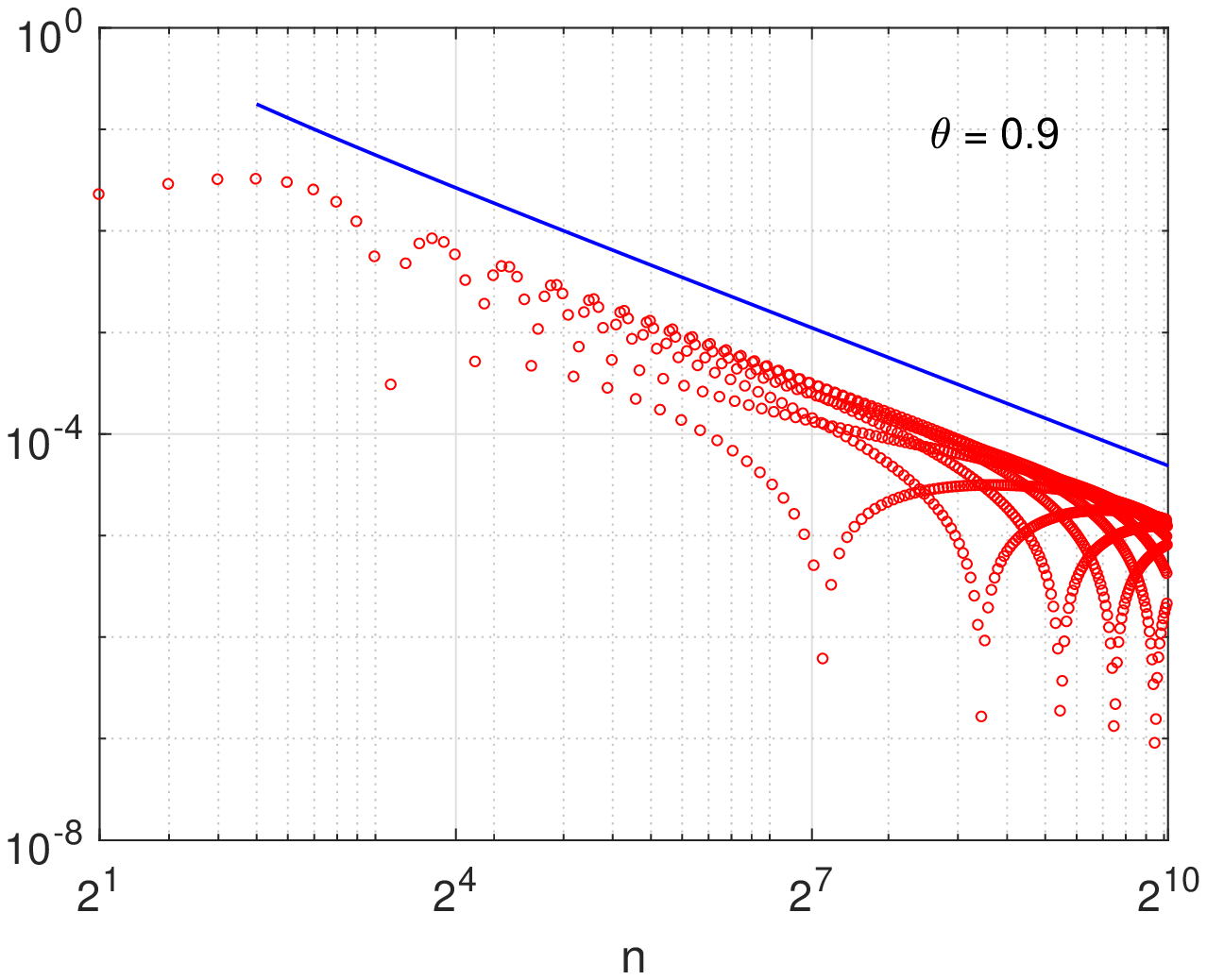}
\end{overpic}
\caption{The exact Legendre coefficients (circles) and the bound
$\mathrm{B}^{\mathrm{New}}(n)$ (line) for $\theta=0.3$ (left),
$\theta=0.6$ (middle) and $\theta=0.9$ (right). Here $f(x)$ is
defined in \eqref{eq:ModFun}.}\label{fig:LegExam1}
\end{figure}
\end{example}

\begin{example}
We consider the truncated power function
\begin{align}
f(x) = (x - \theta)_{+}^{2} = \left\{
\begin{array}{ll}
(x-\theta)^2, & \hbox{$x\geq\theta$,}   \\[6pt]
 0,           & \hbox{$x<\theta$,}
            \end{array}
            \right.
\end{align}
where $\theta\in(-1,1)$. It is clear that $f$ and $f{'}$ are
absolutely continuous on $\Omega$ and $f{''}$ is of bounded
variation on $\Omega$. Moreover, direct calculation gives $V_2=2$,
and thus
\begin{align}
|a_n| \leq \frac{4}{\sqrt{2\pi(n-2)}(n-1/2)(n-3/2)}, \quad
n=3,4,\ldots. \nonumber
\end{align}
In Figure \ref{fig:LegExam2} we illustrate the exact Legendre
coefficients and the above bound for three values of $\theta$.
Clearly, we can see that our new bound is quite sharp whenever
$\theta$ is not close to both endpoints and is slightly
overestimated whenever $\theta\rightarrow\pm1$.
\begin{figure}[h]
\centering
\begin{overpic}
[width=4.8cm,height=5.cm]{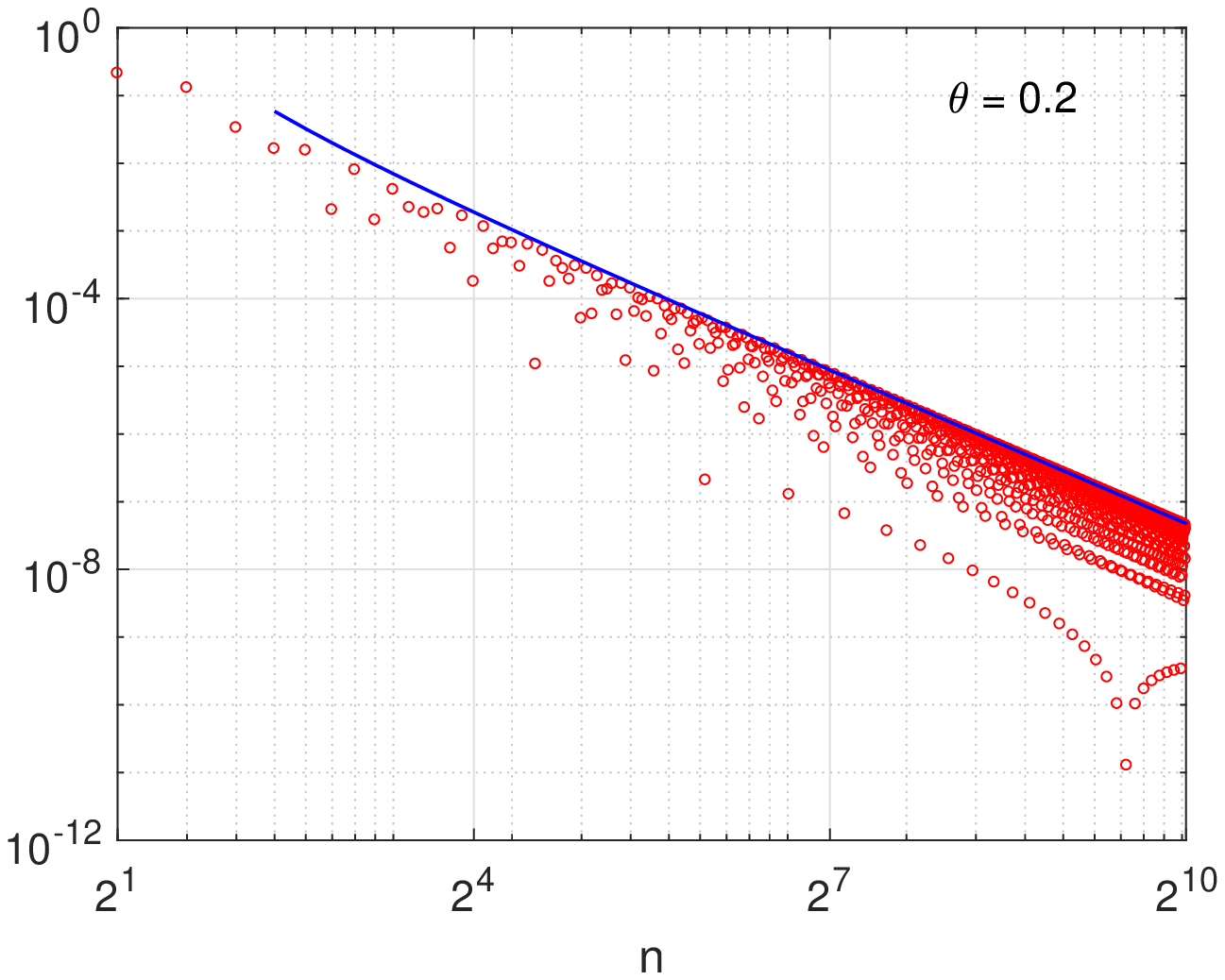}
\end{overpic}
\begin{overpic}
[width=4.8cm,height=5.cm]{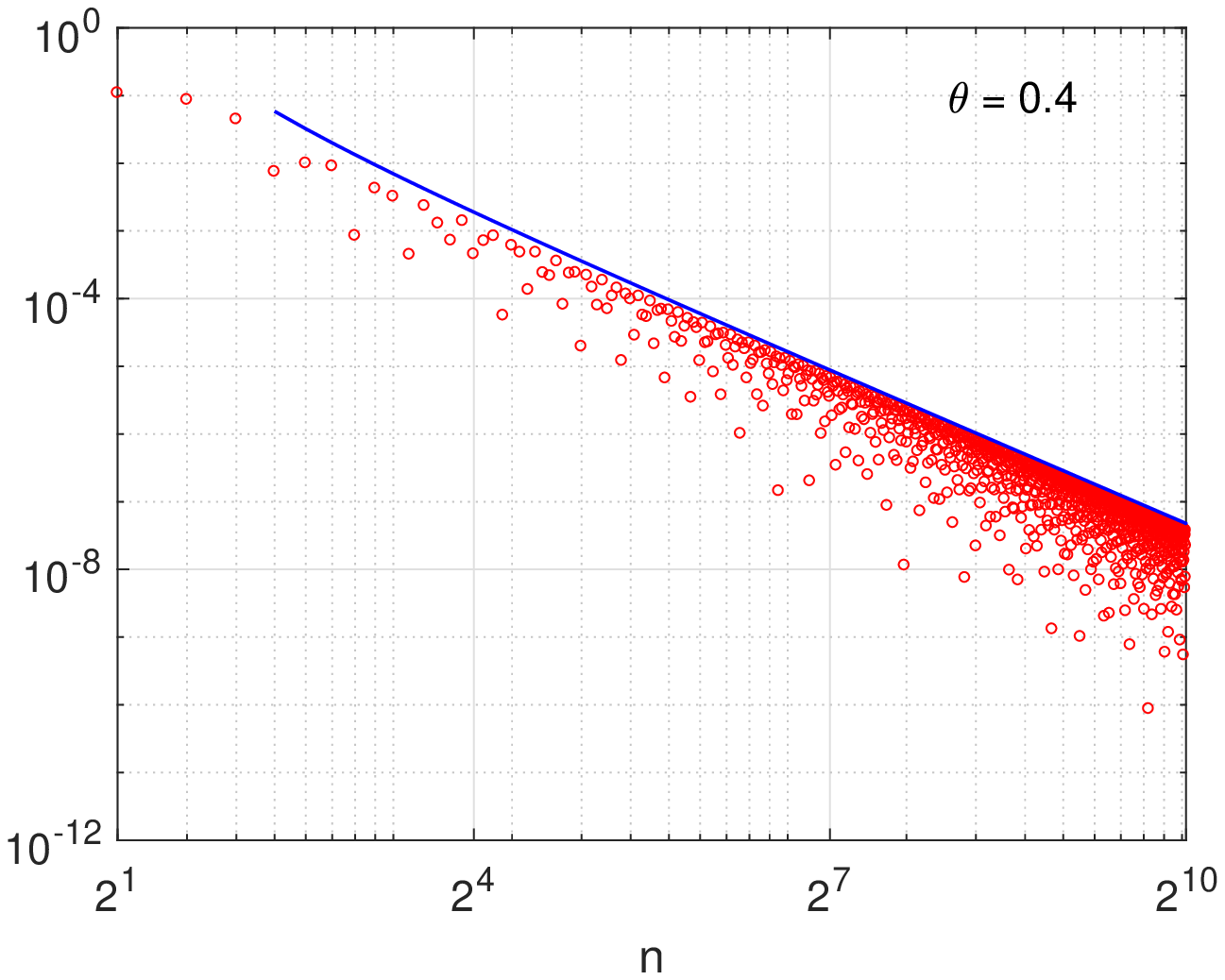}
\end{overpic}
\begin{overpic}
[width=4.8cm,height=5.cm]{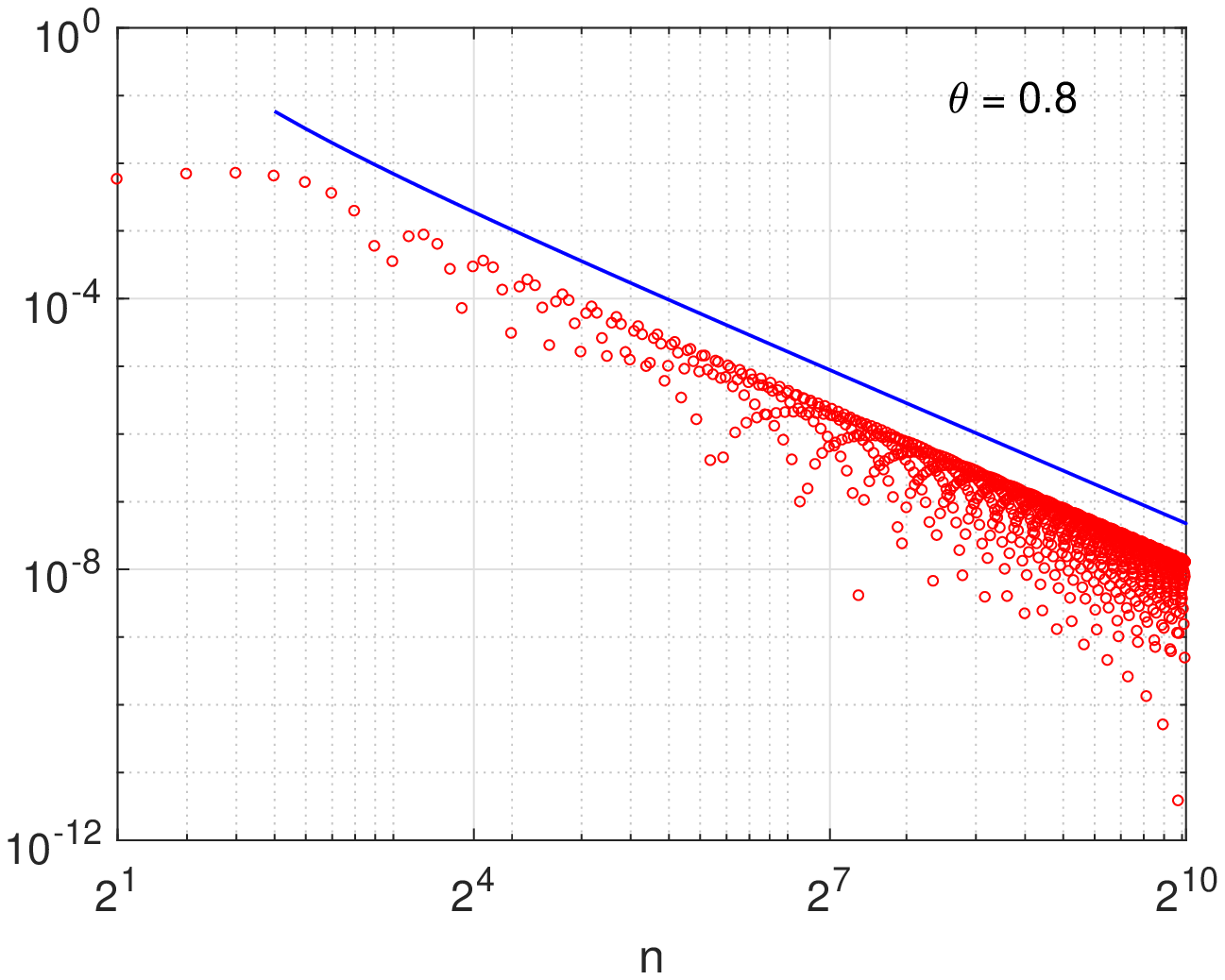}
\end{overpic}
\caption{The exact Legendre coefficients (circles) and the bound
(line) for $\theta=0.2$ (left), $\theta=0.4$ (middle) and
$\theta=0.8$ (right). Here $f(x)$ is defined in (3.6).}
\label{fig:LegExam2}
\end{figure}
\end{example}

In the following, we apply Theorem \ref{thm:NewBound} to establish
some new explicit error bounds for Legendre projections in the $L^2$
and $L^{\infty}$ norms.
\begin{theorem}\label{thm:BoundL2}
If $f,f{'},\ldots,f^{(m-1)}$ are absolutely continuous on $\Omega$
and $f^{(m)}$ is of bounded variation on $\Omega$ for some
nonnegative integer $m$, i.e.,
$V_m=\mathcal{V}(f^{(m)},\Omega)<\infty$.
\begin{itemize}
\item[\rm (i)] For $m=0,1,\ldots$ and $n\geq m+1$,
\begin{align}\label{eq:MeanError}
\|f - f_n\|_{L^2(\Omega)} \leq
\frac{V_m}{\sqrt{\pi(m+1/2)}(n-m)^{m+1/2}},
\end{align}

\item[\rm (ii)] For $m\geq1$ and $n\geq m+1$,
\begin{align}\label{eq:MaxError}
\|f - f_n\|_{L^{\infty}(\Omega)} \leq \left\{
\begin{array}{ll}
{\displaystyle \frac{4V_1}{\sqrt{2\pi(n-1)}}}, & \hbox{$m=1$,}   \\[18pt]
{\displaystyle
\frac{2V_m/(m-1)}{\sqrt{2\pi(n+1-m)}}\prod_{k=1}^{m-1}
\frac{1}{n-k+1/2}}, & \hbox{$m\geq2$,}
            \end{array}
            \right.
\end{align}
and the product is assumed to be one whenever $m=1$.
\end{itemize}
\end{theorem}
\begin{proof}
As for \eqref{eq:MeanError}, recalling the orthogonal property of
Legendre polynomials \eqref{def:LegOrth}, we find that
\begin{align}
\|f - f_n\|_{L^2(\Omega)}^2 &= %\int_{\Omega}\left(f-f_n\right)^2 \mathrm{d}x =
\sum_{k=n+1}^{\infty} (a_k)^2 \left(k + \frac{1}{2} \right)^{-1}.
\nonumber
\end{align}
Furthermore, using Theorem \ref{thm:NewBound} we obtain
\begin{align}
\|f - f_n\|_{L^2(\Omega)}^2 \leq \frac{2V_m^2}{\pi}
\sum_{k=n+1}^{\infty} \frac{1}{(k-m)^{2m+2}} &\leq
\frac{2V_m^2}{\pi}
\int_{n}^{\infty} \frac{1}{(x-m)^{2m+2}}    \nonumber \\
&= \frac{V_m^2}{\pi(m+1/2)(n-m)^{2m+1}} .   \nonumber
\end{align}
Taking the square root on both sides gives \eqref{eq:MeanError}. As
for \eqref{eq:MaxError}, it follows by combining \eqref{eq:NewBound}
in Theorem \ref{thm:NewBound} with the fact that $|P_k(x)|\leq1$.
This ends the proof.
\end{proof}

\begin{figure}[h]
\centering
\begin{overpic}
[width=7cm,height=6.cm]{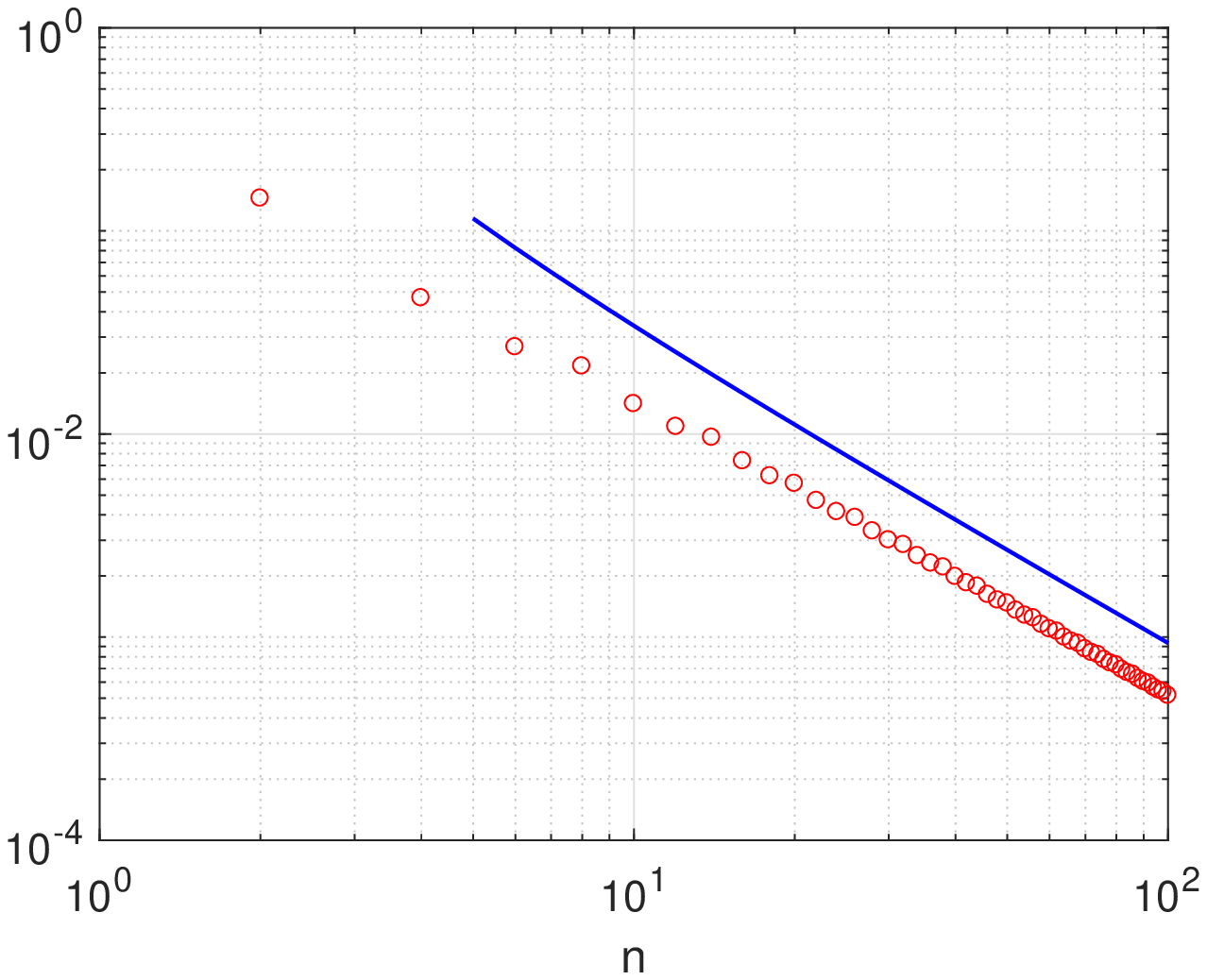}
\end{overpic}
\begin{overpic}
[width=7cm,height=6.cm]{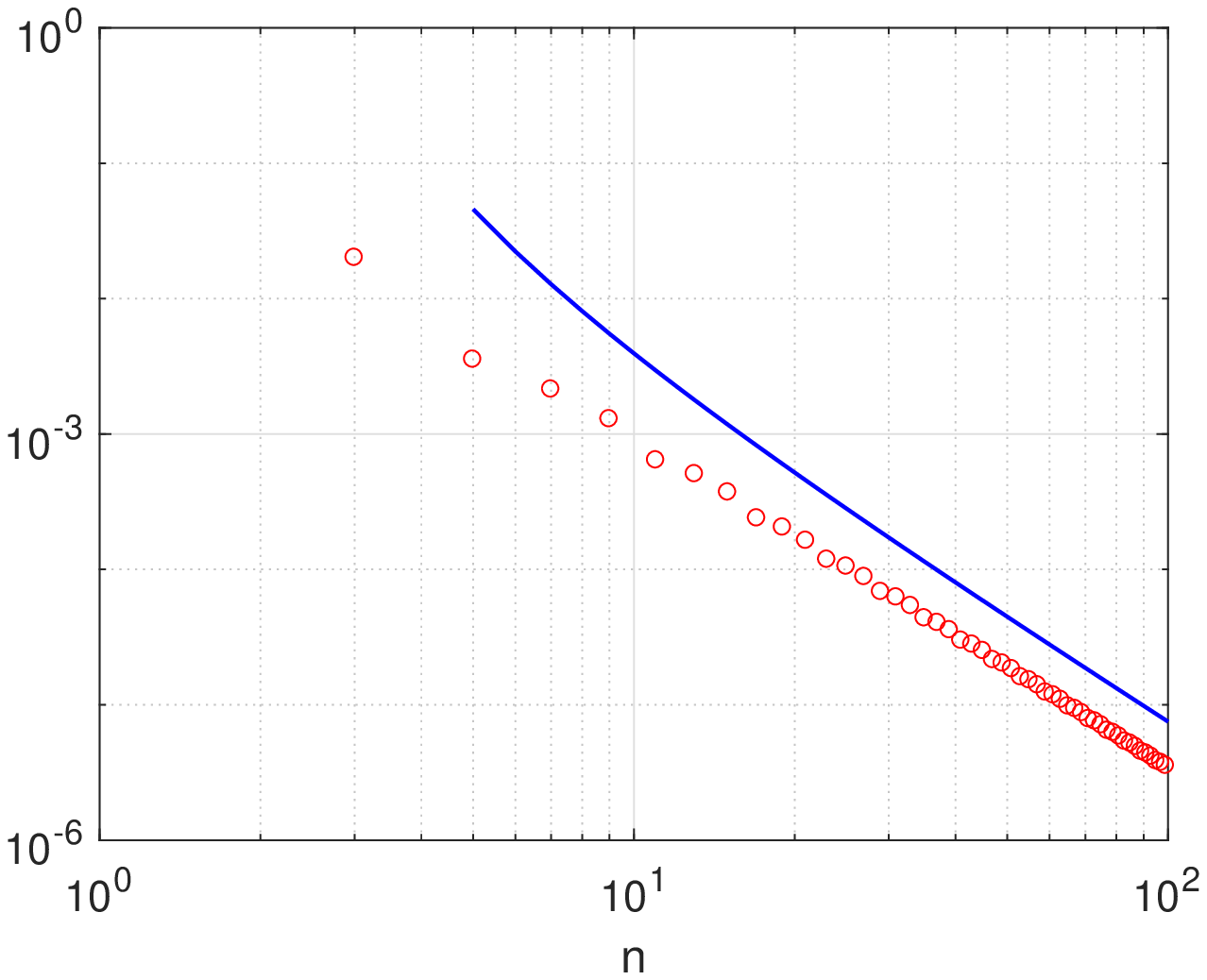}
\end{overpic}
\caption{The error $\|f-f_n\|_{L^2(\Omega)}$ (circles) and the error
bound (line) for $f(x)=|x-0.5|$ (left) and $f(x)=(x-0.5)_{+}^2$
(right).} \label{fig:LegExam3}
\end{figure}

In Figure \ref{fig:LegExam3} we show the actual error of $f_n$ and
the error bound \eqref{eq:MeanError} in the $L^2$ norm as a function
of $n$ for two test functions. Clearly, we can see that the error
bound \eqref{eq:MeanError} is optimal up to a constant factor.

\begin{remark}
Recently, the following error bound was proved in
\cite[Theorem~3.4]{liu2020legendre}
\begin{align}\label{eq:Liu}
\|f - f_n\|_{L^2(\Omega)} \leq \frac{V_m}{\sqrt{\pi(m+1/2)}}
\sqrt{\frac{\Gamma(n-m)}{\Gamma(n+m+1)}}, \quad n\geq m+1.
\end{align}
Comparing \eqref{eq:MeanError} with \eqref{eq:Liu}, it is easily
verified that their ratio asymptotes to one as $n\rightarrow\infty$
and thus both bounds are almost identical for large $n$. Whenever
$n$ is small, direct calculations show that \eqref{eq:Liu} is
slightly sharper than \eqref{eq:MeanError}. On the other hand, it is
easily seen that our bound \eqref{eq:MeanError} is more explicit and
informative.
\end{remark}

\begin{remark}
In the case where $f$ is piecewise analytic on $\Omega$ and has
continuous derivatives up to order $m-1$ for some $m\in\mathbb{N}$,
the current author has proved in \cite[Theorem~3]{wang2021legendre}
that the optimal rate of convergence of $f_n$ is $O(n^{-m})$. For
such functions, however, the predicted rate of convergence by the
error bound \eqref{eq:MaxError} is only $O(n^{-m+1/2})$. Hence, the
error bound \eqref{eq:MaxError} is suboptimal in the sense that it
overestimates the actual error by a factor of $n^{1/2}$.
\end{remark}

We further ask: How to derive an optimal error bound of $f_n$ in the
$L^{\infty}$ norm? In the following, we shall establish a weighted
inequality for the error of $f_n$ in the $L^{\infty}$ norm and an
explicit and optimal error bound for $f_n$ in the $L^{\infty}$ norm
under the condition that the maximum error of $f_n$ is attained in
the interior of $\Omega$.
\begin{theorem}\label{thm:MaxBound}
If $f,f{'},\ldots,f^{(m-1)}$ are absolutely continuous on $\Omega$
and $f^{(m)}$ is of bounded variation on $\Omega$ for some
$m\in\mathbb{N}$, i.e., $V_m=\mathcal{V}(f^{(m)},\Omega)<\infty$.
\begin{itemize}
\item[\rm (i)] For $n\geq m$, we have
\begin{align}\label{eq:MaxBoundW}
\max_{x\in\Omega} (1-x^2)^{1/4}|f(x) - f_n(x)| \leq
\frac{2V_m}{m\pi} \prod_{j=1}^{m} \frac{1}{n-j+1/2}.
\end{align}

\item[\rm (ii)] If the maximum error of $f_n$ is
attained at $\tau\in(-1,1)$ for $n\geq n_0$. Then, for $n\geq
\max\{n_0,m\}$, we have
\begin{align}\label{eq:MaxBoundM}
\|f - f_n\|_{L^{\infty}(\Omega)} \leq \frac{2V_m}{\pi
m(1-\tau^2)^{1/4}} \prod_{j=1}^{m} \frac{1}{n-j+1/2}.
\end{align}
\end{itemize}
\end{theorem}
\begin{proof}
We first consider the proof of \eqref{eq:MaxBoundW}. Recall the
Bernstein-type inequality satisfied by Legendre polynomials (see
\cite{Antonov1981} or \cite[Equation (18.14.7)]{olver2010nist})
\begin{align}\label{eq:Bernstein}
(1-x^2)^{1/4} |P_n(x)| < \sqrt{\frac{2}{\pi}}
\left(n+\frac{1}{2}\right)^{-1/2}, \quad  x\in\Omega,
\end{align}
and the bound on the right hand side is optimal in the sense that
$(n+1/2)^{-1/2}$ can not be improved to $(n+1/2+\epsilon)^{-1/2}$
for any $\epsilon>0$ and the constant $(2/\pi)^{1/2}$ is best
possible. Consequently, using the above inequality and Theorem
\ref{thm:NewBound}, we deduce that
\begin{align}\label{eq:Step1}
(1-x^2)^{1/4} |f(x) - f_n(x)| &\leq \sqrt{\frac{2}{\pi}}
\sum_{k=n+1}^{\infty} |a_k|
\left(k+\frac{1}{2}\right)^{-1/2} \nonumber \\
&\leq \frac{2V_m}{\pi} \sum_{k=n+1}^{\infty}
\frac{1}{\sqrt{(k-m)(k+1/2)}} \prod_{j=1}^{m} \frac{1}{k-j+1/2}
\nonumber \\
&\leq \frac{2V_m}{m\pi} \sum_{k=n+1}^{\infty} \left[ \prod_{j=1}^{m}
\frac{1}{k-j-1/2} - \prod_{j=1}^{m} \frac{1}{k-j+1/2} \right]
\nonumber \\
&= \frac{2V_m}{m\pi} \prod_{j=1}^{m} \frac{1}{n-j+1/2}.
\end{align}
This proves \eqref{eq:MaxBoundW}. As for \eqref{eq:MaxBoundM}, note
that the maximum error of $f_n$ is attained at $x=\tau$, we have
\begin{align}
\|f - f_n\|_{L^{\infty}(\Omega)} = |f(\tau) - f_n(\tau)| \leq
\sum_{k=n+1}^{\infty} |a_k| |P_k(\tau)|. \nonumber
\end{align}
Combining the last inequality with Theorem \ref{thm:NewBound} and
\eqref{eq:Bernstein} and using a similar process as in
\eqref{eq:Step1} gives \eqref{eq:MaxBoundM}. This ends the proof.
\end{proof}

\begin{remark}
For functions with interior singularities, from the pointwise error
analysis developed in \cite{wang2021gegenbauer,wang2021cheby} we
know that the maximum error of $f_n$ is actually determined by the
errors at these interior singularities for moderate and large values
of $n$. In this case, the error bound in \eqref{eq:MaxBoundM} is
optimal in the sense that it can not be improved with respect to $n$
up to a constant factor; see Figure \ref{fig:LegExam4} for an
illustration.
\end{remark}

In Figure \ref{fig:LegExam4} we show the maximum error of $f_n$ and
the error bound \eqref{eq:MaxBoundM} as a function of $n$ for the
functions $f(x)=|x-0.2|$ and $f(x)=(x-0.5)_{+}^2$. For these two
functions, the maximum errors of $f_n$ are attained at $x=0.2$ and
$x=0.5$, respectively, for moderate and large values of $n$. Thus,
the error bound in \eqref{eq:MaxBoundM} can be applied to these two
examples. We see from Figure \ref{fig:LegExam4} that the error bound
\eqref{eq:MaxBoundM} is optimal with respect to $n$ up to a constant
factor.

\begin{figure}[h]
\centering
\begin{overpic}
[width=7cm,height=6.cm]{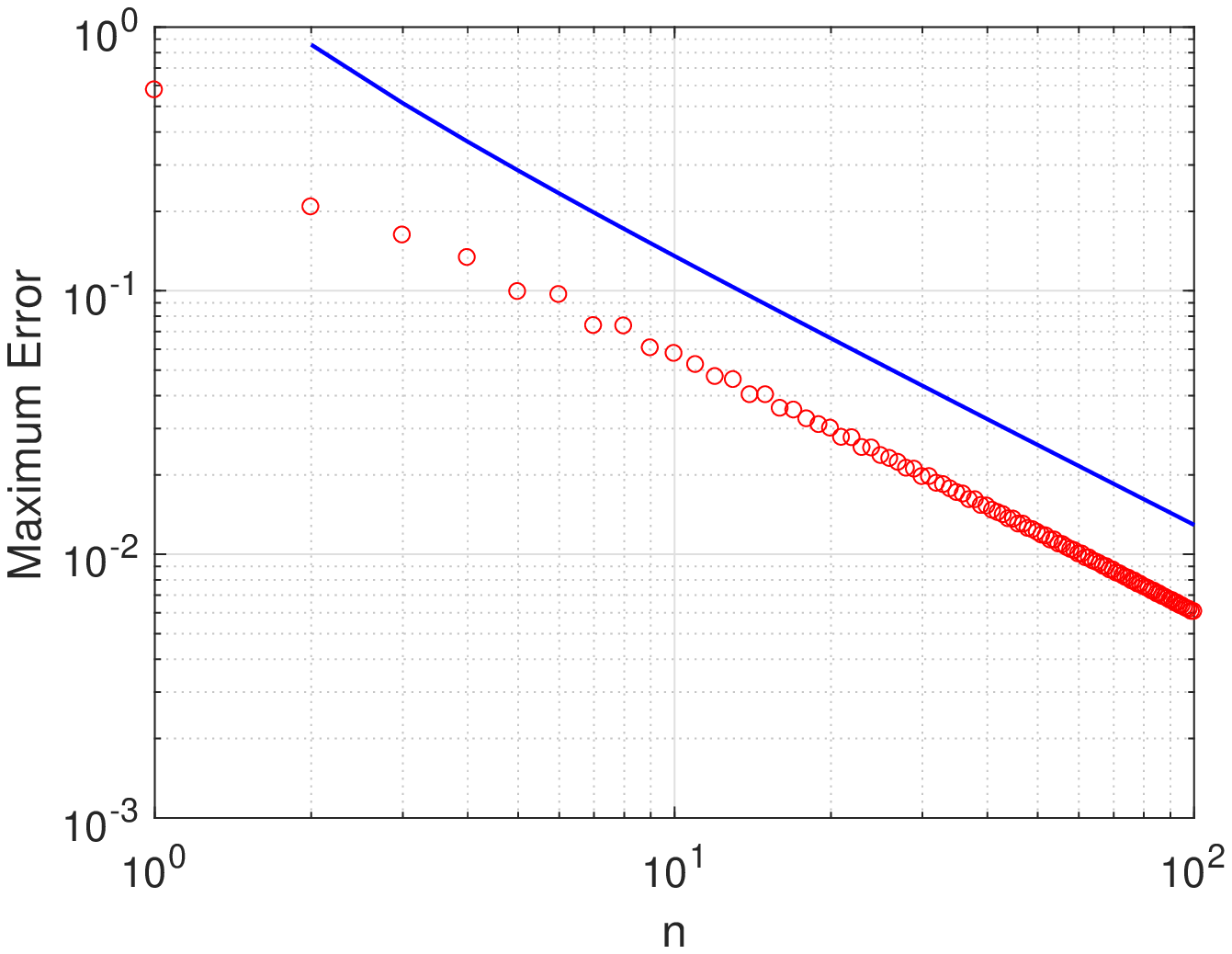}
\end{overpic}
\begin{overpic}
[width=7cm,height=6.cm]{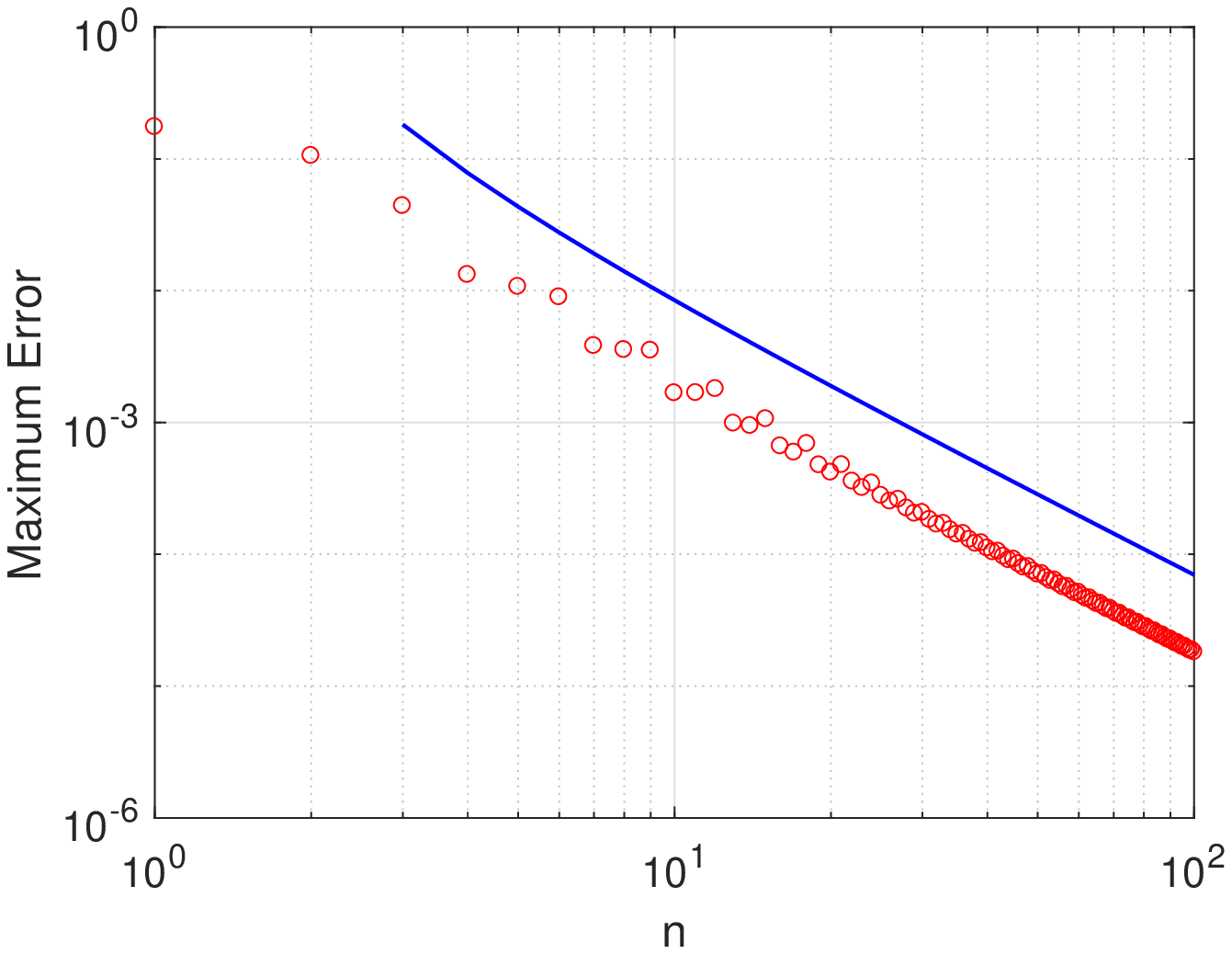}
\end{overpic}
\caption{The maximum error $\|f-f_n\|_{L^{\infty}(\Omega)}$
(circles) and the error bound (line) for $f(x)=|x-0.2|$ (left) and
$f(x)=(x-0.5)_{+}^2$ (right).} \label{fig:LegExam4}
\end{figure}

\section{Extensions}\label{sec:extension}
In this section we present two extensions of our results in Theorem
\ref{thm:PhiFun}, including Gegenbauer-Gauss-Lobatto (GGL) functions
and optimal convergence rates of Legendre-Gauss-Lobatto
interpolation and differentiation for analytic functions.

\subsection{Gegenbauer-Gauss-Lobatto functions}
We introduce the Gegenbauer-Gauss-Lobatto (GGL) functions of the
form
\begin{align}\label{def:GGLFun}
\phi_{n}^{\mathrm{GGL}}(x) &= \frac{n+1}{n+2\lambda}
\omega_{\lambda}(x) C_{n+1}^{\lambda}(x) - \frac{n+2\lambda-1}{n}
\omega_{\lambda}(x) C_{n-1}^{\lambda}(x), \quad n\in\mathbb{N},
\end{align}
where $C_{n}^{\lambda}(x)$ is the Gegenbauer polynomial of degree
$n$ defined in \cite[Equation~(18.5.9)]{olver2010nist} and
$\omega_{\lambda}(x)=(1-x^2)^{\lambda-1/2}$ is the Gegenbauer weight
function. Moreover, by using
\cite[Equation~(18.9.8)]{olver2010nist}, they can also be written as
\begin{align}\label{def:GGLFun2}
\phi_{n}^{\mathrm{GGL}}(x) &=
-\frac{4\lambda(n+\lambda)}{n(n+2\lambda)}\omega_{\lambda+1}(x)C_{n-1}^{\lambda+1}(x).
\end{align}
Notice that $\phi_{n}^{\mathrm{GGL}}(x)=\phi_{n}^{\mathrm{LGL}}(x)$
whenever $\lambda=1/2$ and thus $\phi_{n}^{\mathrm{GGL}}(x)$ can be
viewed as a generalization of $\phi_{n}^{\mathrm{LGL}}(x)$. We are
now ready to prove the following theorem.
\begin{theorem}\label{thm:GGLFun}
Let $\phi_{n}^{\mathrm{GGL}}(x)$ be the function defined in
\eqref{def:GGLFun} or \eqref{def:GGLFun2} and let $\lambda>-1/2$ and
$\lambda\neq0$.
\begin{itemize}
\item[\rm (i)] $|\phi_{n}^{\mathrm{GGL}}(x)|$ is even and $\phi_{n}^{\mathrm{GGL}}(\pm1)=0$ for all $n\in\mathbb{N}$.

\item[\rm (ii)] For all $n\in\mathbb{N}$, the derivative of $\phi_{n}^{\mathrm{GGL}}(x)$ is
\begin{align}\label{eq:GGLD}
\frac{\mathrm{d}}{\mathrm{d}x}\phi_{n}^{\mathrm{GGL}}(x) =
2(n+\lambda)\omega_{\lambda}(x)C_{n}^{\lambda}(x).
\end{align}

\item[\rm (iii)] For all $n\in\mathbb{N}$, the differential recurrence relation of
$\phi_n^{\mathrm{GGL}}(x)$ is
\begin{align}\label{eq:GGLDiffRec}
\phi_{n}^{\mathrm{GGL}}(x) &= \frac{\mathrm{d}}{\mathrm{d}x} \left[
\frac{(n+1)}{2(n+\lambda+1)(n+2\lambda)}
\phi_{n+1}^{\mathrm{GGL}}(x) - \frac{(n+2\lambda-1)
}{2n(n+\lambda-1)} \phi_{n-1}^{\mathrm{GGL}}(x) \right].
\end{align}

\item[\rm (iv)] Let $\nu=\lfloor (n+1)/2 \rfloor$ for all $n\in\mathbb{N}$ and let
$x_{\nu}<\cdots<x_1$ be the zeros of $C_n^{\lambda}(x)$ on the
interval $[0,1]$. Then, $|\phi_{n}^{\mathrm{GGL}}(x)|$ attains its
local maximum values at these points $\{x_k\}_{k=1}^{\nu}$ and
\begin{align}\label{eq:GGLDecSeqI}
|\phi_{n}^{\mathrm{GGL}}(x_1)|<|\phi_{n}^{\mathrm{GGL}}(x_2)|<\cdots<|\phi_{n}^{\mathrm{GGL}}(x_{\nu})|,
\end{align}
for $\lambda>0$ and
\begin{align}\label{eq:GGLDecSeqII}
|\phi_{n}^{\mathrm{GGL}}(x_1)|>|\phi_{n}^{\mathrm{GGL}}(x_2)|>\cdots>|\phi_{n}^{\mathrm{GGL}}(x_{\nu})|,
\end{align}
for $\lambda<0$.

\item[\rm (v)] For $\lambda>0$, the maximum value of $|\phi_{n}^{\mathrm{GGL}}(x)|$ satisfies
\begin{align}\label{eq:GGLIneq}
\max_{x\in\Omega}|\phi_{n}^{\mathrm{GGL}}(x)| \leq
\mathcal{B}_n^{\lambda} := \left\{
\begin{array}{ll}
{\displaystyle \frac{4(n+\lambda)}{n(n+2\lambda)}
\frac{\Gamma(\frac{n+1}{2}+\lambda)}{\Gamma(\frac{n+1}{2})\Gamma(\lambda)}},  & \hbox{$n=1,3,\ldots$,}   \\[15pt]
{\displaystyle \frac{2(n+\lambda)}{\sqrt{n(n+2\lambda)}}
\frac{\Gamma(\frac{n}{2}+\lambda)}{\Gamma(\frac{n+2}{2})\Gamma(\lambda)}},
& \hbox{$n=2,4,\ldots$.}
            \end{array}
            \right.
\end{align}
For $\lambda<0$, the maximum value of $|\phi_{n}^{\mathrm{GGL}}(x)|$
satisfies
\begin{align}
\max_{x\in\Omega}|\phi_{n}^{\mathrm{GGL}}(x)| =
|\phi_{n}^{\mathrm{GGL}}(x_1)| = O(n^{-1}), \quad n\gg1.
\end{align}
\end{itemize}
\end{theorem}
\begin{proof}
As for (i), the assertion $|\phi_{n}^{\mathrm{GGL}}(x)|$ is even
follows from the symmetry of Gegenbauer polynomials (i.e.,
$C_k^{\lambda}(-x)=(-1)^kC_k^{\lambda}(x)$ for all $k=0,1,\ldots$)
and $\phi_{n}^{\mathrm{GGL}}(\pm1)=0$ follows from
\eqref{def:GGLFun2}. As for (ii), from
\cite[Equation~(18.9.20)]{olver2010nist} we know that
\begin{align}\label{eq:Deriv}
\frac{\mathrm{d}}{\mathrm{d}x} \left(\omega_{\lambda+1}(x)
C_n^{\lambda+1}(x) \right) =
-\frac{(n+1)(n+2\lambda+1)}{2\lambda}\omega_{\lambda}(x)
C_{n+1}^{\lambda}(x).
\end{align}
The combination of \eqref{eq:Deriv} and \eqref{def:GGLFun2} proves
\eqref{eq:GGLD}. As for (iii), it is a direct consequence of
\eqref{def:GGLFun} and \eqref{eq:GGLD}. Now we consider the proof of
(iv). Since $|\phi_{n,\lambda}^{\mathrm{GGL}}(x)|$ is even, we only
need to consider the maximum values of
$|\phi_{n,\lambda}^{\mathrm{GGL}}(x)|$ at the nonnegative zeros of
$C_n^{\lambda}(x)$. Similar to the argument of Theorem
\ref{thm:PhiFun}, we introduce the following auxiliary function
\begin{align}\label{eq:AuGFun}
\psi(x) = \frac{n(n+2\lambda)}{4(n+\lambda)^2}
\left(\phi_{n}^{\mathrm{GGL}}(x)\right)^2 + (1-x^2)
(\omega_{\lambda}(x) C_n^{\lambda}(x))^2.
\end{align}
Combining \eqref{def:GGLFun2}, \eqref{eq:GGLD} and \eqref{eq:Deriv}
and after some calculations, we get
\begin{align}\label{eq:AuGFun}
\psi{'}(x) &= \frac{n(n+2\lambda)}{(n+\lambda)}
\phi_{n}^{\mathrm{GGL}}(x)\omega_{\lambda}(x)C_{n}^{\lambda}(x)
-2x (\omega_{\lambda}(x) C_n^{\lambda}(x))^2 \nonumber \\
&~~~~~~~~ + 2(1-x^2)\omega_{\lambda}(x) C_n^{\lambda}(x)
\frac{\mathrm{d}}{\mathrm{d}x} (\omega_{\lambda}(x)
C_n^{\lambda}(x)) \nonumber \\
&= 2\omega_{\lambda}(x) C_n^{\lambda}(x) \left[
\frac{n(n+2\lambda)}{2(n+\lambda)}\phi_{n}^{\mathrm{GGL}}(x) -
x\omega_{\lambda}(x) C_n^{\lambda}(x) \right. \nonumber \\
&~~~~~~~~ \left. + (1-x^2)\frac{\mathrm{d}}{\mathrm{d}x} (
\omega_{\lambda}(x) C_n^{\lambda}(x)) \right] \nonumber\\
&= 2\omega_{\lambda}(x) C_n^{\lambda}(x) \bigg[ -2\lambda
\omega_{\lambda+1}(x) C_{n-1}^{\lambda+1}(x)  -
x\omega_{\lambda}(x) C_n^{\lambda}(x) \nonumber \\
&~~~~~~~~ \left. -\frac{(n+1)(n+2\lambda-1)}{2(n+\lambda)}
\omega_{\lambda}(x) (C_{n+1}^{\lambda}(x) - C_{n-1}^{\lambda}(x))
\right],
\end{align}
where we have used \cite[Equation~(18.9.7)]{olver2010nist} in the
last step. Furthermore, invoking
\cite[Equation~(18.9.8)]{olver2010nist} and
\cite[Equation~(18.9.1)]{olver2010nist}, and after some lengthy but
elementary calculations, we obtain that
\begin{align}
\psi{'}(x) = -4\lambda x (\omega_{\lambda}(x) C_n^{\lambda}(x))^2.
\end{align}
It is easily seen that $\psi(x)$ is strictly decreasing on $[0,1]$
whenever $\lambda>0$ and is strictly increasing on $[0,1]$ whenever
$\lambda<0$, and thus the inequalities \eqref{eq:GGLDecSeqI} and
\eqref{eq:GGLDecSeqII} follow immediately. As for (v), we first
consider the case $\lambda>0$. From (iv) we infer that
\begin{align}
\max_{x\in\Omega}|\phi_{n}^{\mathrm{GGL}}(x)| \leq
|\phi_{n}^{\mathrm{GGL}}(x_{\nu})|.  \nonumber
\end{align}
In the case when $n$ is odd, it is easily seen that $x_{\nu}=0$ and
thus
\begin{align}
\max_{x\in\Omega}|\phi_{n}^{\mathrm{GGL}}(x)| \leq
|\phi_{n}^{\mathrm{GGL}}(0)| = \frac{4(n+\lambda)}{n(n+2\lambda)}
\frac{\Gamma(\frac{n+1}{2}+\lambda)}{\Gamma(\frac{n+1}{2})
\Gamma(\lambda)}. \nonumber
\end{align}
This proves the case of odd $n$. In the case when $n$ is even,
notice that $\psi(x)$ is strictly decreasing on the interval $[0,1]$
we obtain
\begin{align}
\psi(x)\leq \psi(0) \quad \Longrightarrow \quad
\max_{x\in\Omega}|\phi_{n}^{\mathrm{GGL}}(x)| \leq
\frac{2(n+\lambda)}{\sqrt{n(n+2\lambda)}}
\frac{\Gamma(\frac{n}{2}+\lambda)}{\Gamma(\frac{n+2}{2})
\Gamma(\lambda)}. \nonumber
\end{align}
This proves the case of even $n$. Finally, we consider the case of
$\lambda<0$. On the one hand, from \eqref{eq:GGLDecSeqII} we obtain
immediately that $\max_{x\in\Omega}|\phi_{n}^{\mathrm{GGL}}(x)| =
|\phi_{n}^{\mathrm{GGL}}(x_1)|$. On the other hand, from
\cite[Theorem~8.9.1]{szego1975orth} we obtain that for $n\gg1$ that
\begin{align}
x_1 = \cos\left(\frac{\pi}{n} + O(1) \right), \quad
\frac{\mathrm{d}}{\mathrm{d}x}C_n^{\lambda}(x)\big|_{x=x_1} =
O(n^{2\lambda+1}). \nonumber
\end{align}
Combining these with \eqref{def:GGLFun2} gives the desired estimate.
This completes the proof.
\end{proof}

As a direct consequence of Theorem \ref{thm:GGLFun}, we have the
following corollary.
\begin{corollary}
For $\lambda\geq1$, we have
\begin{align}\label{eq:MaxNew}
\max_{x\in\Omega}|\omega_{\lambda}(x)C_{n}^{\lambda}(x)| &\leq
\left\{
\begin{array}{ll}
{\displaystyle \frac{1}{\Gamma(\lambda)}\frac{\Gamma(\frac{n}{2}+\lambda)}{\Gamma(\frac{n}{2}+1)}},  & \hbox{$n=0,2,4,\ldots$,}   \\[15pt]
{\displaystyle \sqrt{\frac{n+2\lambda-1}{n+1}}
\frac{\Gamma(\frac{n-1}{2}+\lambda)}{\Gamma(\lambda)\Gamma(\frac{n+1}{2})}},
& \hbox{$n=1,3,5,\ldots$.}
            \end{array}
            \right.
\end{align}
\end{corollary}
\begin{proof}
It follows by combining \eqref{def:GGLFun2} and \eqref{eq:GGLIneq}.
\end{proof}
\begin{remark}
Based on a Nicholson-type formula for Gegenbauer polynomials, Durand
proved the following inequality for $\lambda\geq1$ (see
\cite[Equation~(19)]{durand1975})
\begin{align}\label{eq:MaxBound}
\max_{x\in\Omega}|\omega_{\lambda}(x)C_{n}^{\lambda}(x)| &\leq
\frac{1}{\Gamma(\lambda)}
\frac{\Gamma(\frac{n}{2}+\lambda)}{\Gamma(\frac{n}{2}+1)}, \quad
n=0,1,2,\ldots.
\end{align}
Comparing \eqref{eq:MaxBound} with \eqref{eq:MaxNew}, it is easily
seen that both bounds are the same whenever $n$ is even. In the case
of odd $n$, however, direct calculations show that
\[
\sqrt{\frac{n+2\lambda-1}{n+1}}
\frac{\Gamma(\frac{n-1}{2}+\lambda)}{\Gamma(\frac{n+1}{2})} \leq
\frac{\Gamma(\frac{n}{2}+\lambda)}{\Gamma(\frac{n}{2}+1)}, \quad
n=0,1,\ldots,
\]
and thus we have derived an improved bound for
$\max_{x\in\Omega}|\omega_{\lambda}(x)C_{n}^{\lambda}(x)|$ whenever
$n$ is odd.
\end{remark}

\begin{remark}
Making use of the asymptotic expansion of the ratio of gamma
functions (see, e.g., \cite[Equation~(5.11.13)]{olver2010nist}), it
follows that
\begin{align}
\lim_{n\rightarrow\infty} n^{1-\lambda} \mathcal{B}_n^{\lambda} =
\frac{2^{2-\lambda}}{\Gamma(\lambda)}.  % \quad  \Longrightarrow  \quad
%\eta_n \leq \mathcal{K} \frac{2^{2-\lambda}}{\Gamma(\lambda)}
%n^{\lambda-1},
\end{align}
Clearly, we see that $\mathcal{B}_n^{\lambda}=O(n^{\lambda-1})$ for
large $n$. Direct calculations show that the sequences
$\{\mathcal{B}_{2k-1}^{\lambda}\}_{k=1}^{\infty}$ and
$\{\mathcal{B}_{2k}^{\lambda}\}_{k=1}^{\infty}$ are strictly
decreasing whenever $0<\lambda\leq1$ and the sequences
$\{\mathcal{B}_{2k-1}^{\lambda}\}_{k=2}^{\infty}$ and
$\{\mathcal{B}_{2k}^{\lambda}\}_{k=1}^{\infty}$ are strictly
increasing whenever $\lambda\geq1.2$. For $\lambda\in(1,1.2)$, there
exists a positive integer $\varrho_{\lambda}\in\mathbb{N}$ such that
the sequences
$\{\mathcal{B}_{2k-1}^{\lambda}\}_{k\geq\varrho_{\lambda}}$ and
$\{\mathcal{B}_{2k}^{\lambda}\}_{k\geq \varrho_{\lambda}}$ are
strictly increasing.
\end{remark}

\subsection{Optimal convergence rates of Legendre-Gauss-Lobatto interpolation and differentiation
for analytic functions} Legendre-Gauss-Lobatto interpolation is
widely used in the numerical solution of differential and integral
equations (see, e.g.,
\cite{canuto2006spectral,ern2021,shen2011spectral}). Let $p_n(x)$ be
the unique polynomial which interpolates $f(x)$ at the zeros of
$\phi_n^{\mathrm{LGL}}(x)$ and it is well known that $p_n(x)$ is the
LGL interpolant of degree $n$. If $f$ is analytic inside and on the
ellipse $\mathcal{E}_{\rho}$ for some $\rho>1$, convergence rates of
LGL interpolation and differentiation in the maximum norm have been
thoroughly studied in \cite{xie2013exp} with the help of the
Hermite's contour integral. In particular, the following error
bounds were proved (setting $\lambda=1/2$ in
\cite[Theorem~4.3]{xie2013exp}):
\begin{align}\label{eq:XieI}
\|f-p_n\|_{L^{\infty}(\Omega)} \leq \mathcal{K}
\frac{M(\rho)\sqrt{\pi(\rho^2+\rho^{-2})}(1+\rho^{-2})^{3/2}}{(1-\rho^{-1})^2(\rho-\rho^{-1})^2}
\frac{n^{3/2}}{\rho^{n}},
\end{align}
and
\begin{align}\label{eq:XieII}
\max_{0\leq j\leq n} \left| f{'}(x_j) - p_n{'}(x_j) \right| \leq
\mathcal{K}
\frac{M(\rho)\sqrt{\pi(\rho^2+\rho^{-2})}(1+\rho^{-2})^{3/2}}{4(1-\rho^{-1})^2(\rho-\rho^{-1})^2}
\frac{n^{7/2}}{\rho^{n}},
\end{align}
where $\mathcal{K}\approx1$ is a generic positive constant and
$\{x_j\}_{j=0}^{n}$ are LGL points (i.e., the zeros of
$\phi_n^{\mathrm{LGL}}(x)$). It is clear to see that the above
results imply that the rate of convergence of $p_n(x)$ in the
$L^{\infty}$ norm is $O(n^{3/2}\rho^{-n})$ and the maximum error of
LGL spectral differentiation is $O(n^{7/2}\rho^{-n})$.

In the following, we shall improve the results \eqref{eq:XieI} and
\eqref{eq:XieII} by using Theorem \ref{thm:PhiFun} and show that the
factor $n^{3/2}$ in \eqref{eq:XieI} can actually be removed and the
factor $n^{7/2}$ in \eqref{eq:XieII} can be improved to $n^{3/2}$.
We state our main results in the following theorem.
\begin{theorem}\label{thm:LGLInterpDiff}
If $f$ is analytic inside and on the ellipse $\mathcal{E}_{\rho}$
for some $\rho>1$, then
\begin{align}\label{eq:LGLinterpB}
\|f - p_n\|_{L^{\infty}(\Omega)} \leq \frac{\mathcal{K} \sqrt{2}
M(\rho) L(\mathcal{E}_{\rho})}{\mathrm{d}(\Omega,\mathcal{E}_{\rho})
\pi\sqrt{\rho^2-1}} \frac{1}{\rho^n}.
\end{align}
and
\begin{align}\label{eq:LGLDiffB}
\max_{0\leq j\leq n} \left| f{'}(x_j) - p_n{'}(x_j) \right| &\leq
\frac{\mathcal{K} \sqrt{2} M(\rho)
L(\mathcal{E}_{\rho})}{\mathrm{d}(\Omega,\mathcal{E}_{\rho})
\sqrt{\pi(\rho^2-1)}} \frac{n^{3/2}}{\rho^n},
\end{align}
where $\mathcal{K}$ is a generic positive constant and
$\mathcal{K}\approx1$ for $n\gg1$ and
$\mathrm{d}(\Omega,\mathcal{E}_{\rho})$ denotes the distance from
$\Omega$ to $\mathcal{E}_{\rho}$.
\end{theorem}
\begin{proof}
From the Hermite integral formula
\cite[Theorem~3.6.1]{davis1975interp} we know that the remainder of
the LGL interpolants can be written as
\begin{align}\label{eq:HCI}
f(x) - p_n(x) = \frac{1}{2\pi i}\oint_{\mathcal{E}_{\rho}}
\frac{\phi_n^{\mathrm{LGL}}(x) f(z)}{\phi_n^{\mathrm{LGL}}(z)(z-x)}
\mathrm{d}z,
\end{align}
from which we can deduce immediately that
\begin{align}\label{eq:LGLinterp}
\|f - p_n\|_{L^{\infty}(\Omega)} &\leq
\frac{\max_{x\in\Omega}|\phi_n^{\mathrm{LGL}}(x)|}{\min_{z\in\mathcal{E}_{\rho}}|\phi_n^{\mathrm{LGL}}(z)|}
\times \frac{M(\rho) L(\mathcal{E}_{\rho})}{2\pi
\mathrm{d}(\Omega,\mathcal{E}_{\rho})} \nonumber \\
&= \frac{1}{\min_{z\in\mathcal{E}_{\rho}}|\phi_n^{\mathrm{LGL}}(z)|}
\times \frac{2 M(\rho) L(\mathcal{E}_{\rho})}{\pi
\mathrm{d}(\Omega,\mathcal{E}_{\rho}) \sqrt{2n\pi}},
\end{align}
where we used \eqref{eq:Ineq} in the last step. Moreover, combining
\eqref{eq:HCI} with \eqref{eq:LegRec} we obtain
\begin{align}\label{eq:LGLDiff}
\max_{0\leq j\leq n} \left| f{'}(x_j) - p_n{'}(x_j) \right| &\leq
(2n+1) \frac{\max_{0\leq j\leq
n}|P_n(x_j)|}{\min_{z\in\mathcal{E}_{\rho}}|\phi_n^{\mathrm{LGL}}(z)|}
\times \frac{M(\rho) L(\mathcal{E}_{\rho})}{2\pi
\mathrm{d}(\Omega,\mathcal{E}_{\rho})} \nonumber \\
&=
\frac{(2n+1)}{\min_{z\in\mathcal{E}_{\rho}}|\phi_n^{\mathrm{LGL}}(z)|}
\times \frac{M(\rho) L(\mathcal{E}_{\rho})}{2\pi
\mathrm{d}(\Omega,\mathcal{E}_{\rho})},
\end{align}
where we have used the fact that $|P_k(x)|\leq1$ in the last step.
In order to establish sharp error bounds for the LGL interpolation
and differentiation, it is necessary to find the minimum value of
$|\phi_n^{\mathrm{LGL}}(z)|$ for $z\in\mathcal{E}_{\rho}$. Owing to
\cite[Equation~(5.14)]{wang2018jacobi} we infer that
\begin{align}
P_n(z) = \frac{u^n}{\sqrt{n\pi(1-u^{-2})}} \left[1 + \frac{1}{4n}
\left(\frac{1}{u^2-1} - \frac{1}{2}\right) + O(n^{-2}) \right],
\quad n\gg1, \nonumber
\end{align}
and, after some calculations, we obtain that
\begin{align}\label{eq:LGLMin}
\min_{z\in\mathcal{E}_{\rho}}|\phi_n^{\mathrm{LGL}}(z)| \geq
\mathcal{K} \rho^n \sqrt{\frac{\rho^2-1}{n\pi}},
\end{align}
and $\mathcal{K}\approx1$ for $n\gg1$. Combining this with
\eqref{eq:LGLinterp} and \eqref{eq:LGLDiff} gives the desired
results. This ends the proof.
\end{proof}

\begin{remark}
It is easily seen that the rate of convergence of $p_n(x)$ in the
$L^{\infty}$ norm is $O(\rho^{-n})$, and thus we have improved the
existing result \eqref{eq:XieI}. Moreover, the rate of convergence
of LGL spectral differentiation is $O(n^{3/2}\rho^{-n})$, and thus
we have improved the existing result \eqref{eq:XieII}.
\end{remark}

In the proof of Theorem \ref{thm:LGLInterpDiff}, we have used an
asymptotic estimate of the minimum value of
$|\phi_n^{\mathrm{LGL}}(z)|$ for $z\in\mathcal{E}_{\rho}$. Now we
provide a more detailed observation on this issue. By parameterizing
the ellipse with $z=(\rho e^{i\theta} + (\rho e^{i\theta})^{-1})/2$
with $\rho>1$ and $0\leq\theta<2\pi$, we plot
$|\phi_n^{\mathrm{LGL}}(z)|$ in Figure \ref{fig:Ellipse} for several
values of $\rho$ and $n$. Clearly, we observe that the minimum value
of $|\phi_n^{\mathrm{LGL}}(z)|$ is always attained at
$\theta=0,\pi$. This observation inspires us to raise the following
conjecture:

\vspace{.2cm}

\noindent{\bf Conjecture}: For $n\in\mathbb{N}$ and $\rho>1$,
\begin{align}\label{eq:conjecture}
\min_{z\in\mathcal{E}_{\rho}}\left| \phi_n^{\mathrm{LGL}}(z) \right|
= \left| \phi_n^{\mathrm{LGL}}(z_0) \right|,
\end{align}
where $z_0=\pm(\rho+\rho^{-1})/2$.

\vspace{.2cm}
\begin{figure}[h]
\centering
\begin{overpic}
[width=4.8cm,height=5.0cm]{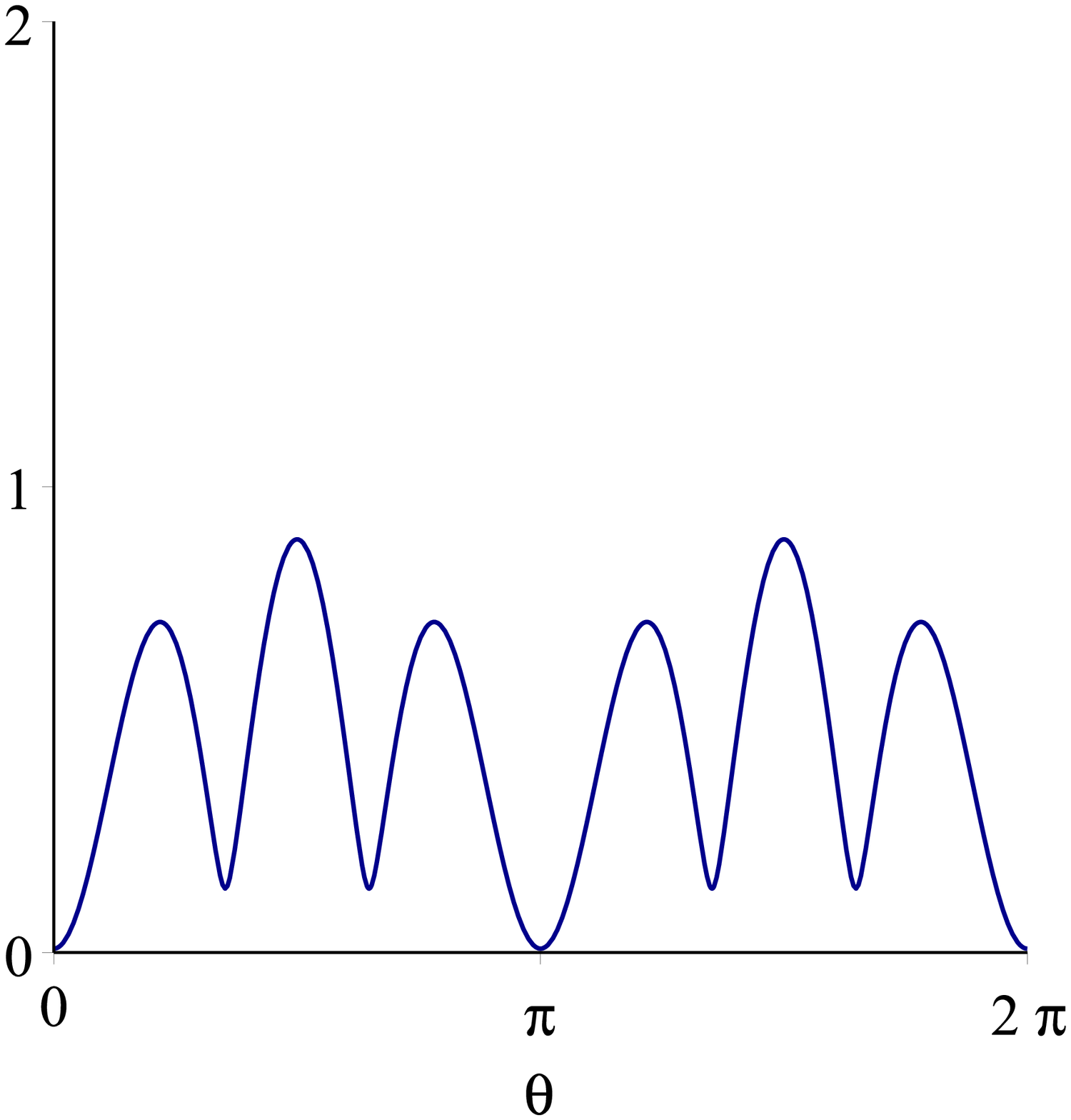}
\end{overpic}
\begin{overpic}
[width=4.8cm,height=5.0cm]{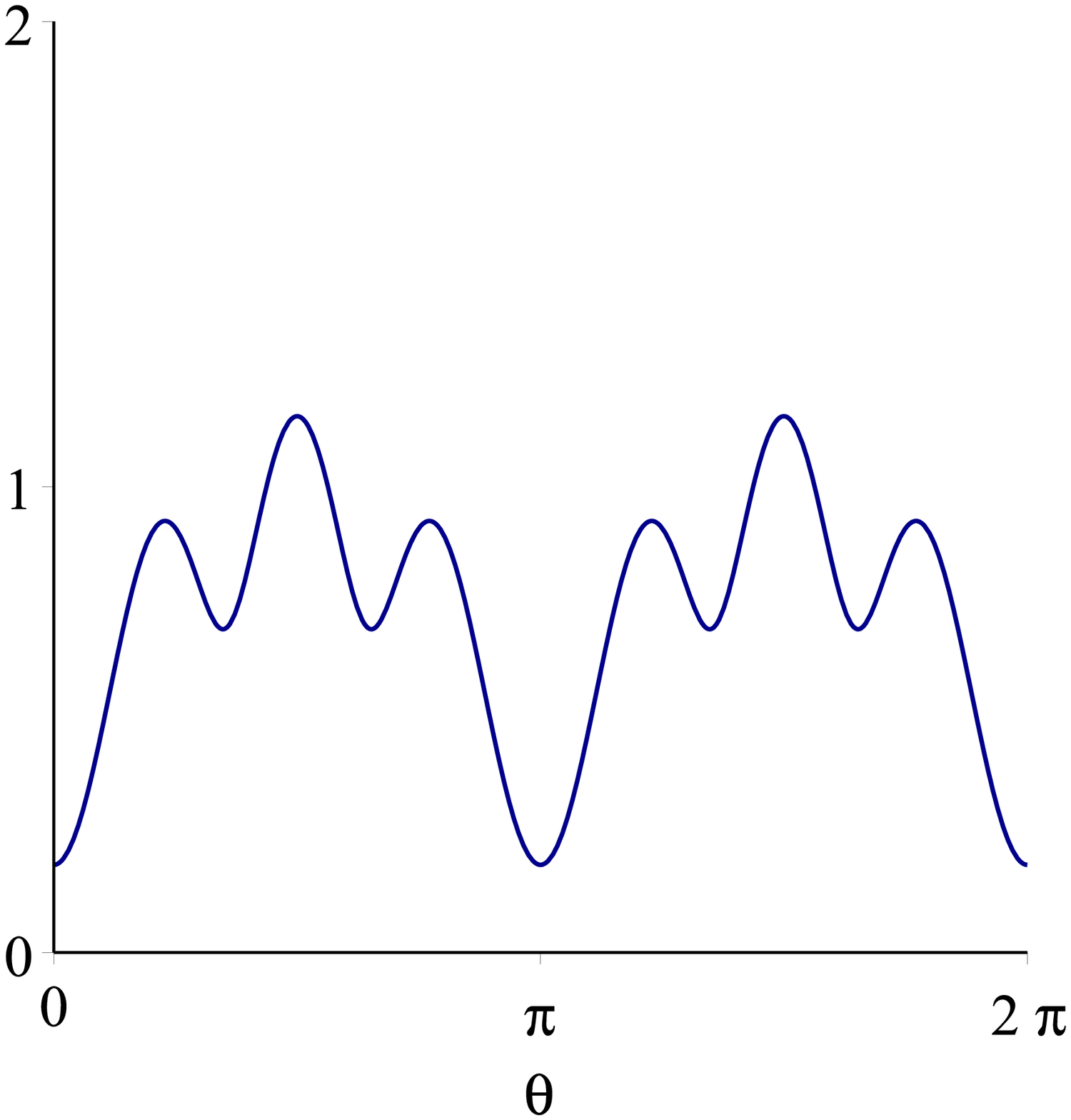}
\end{overpic}
\begin{overpic}
[width=4.8cm,height=5.0cm]{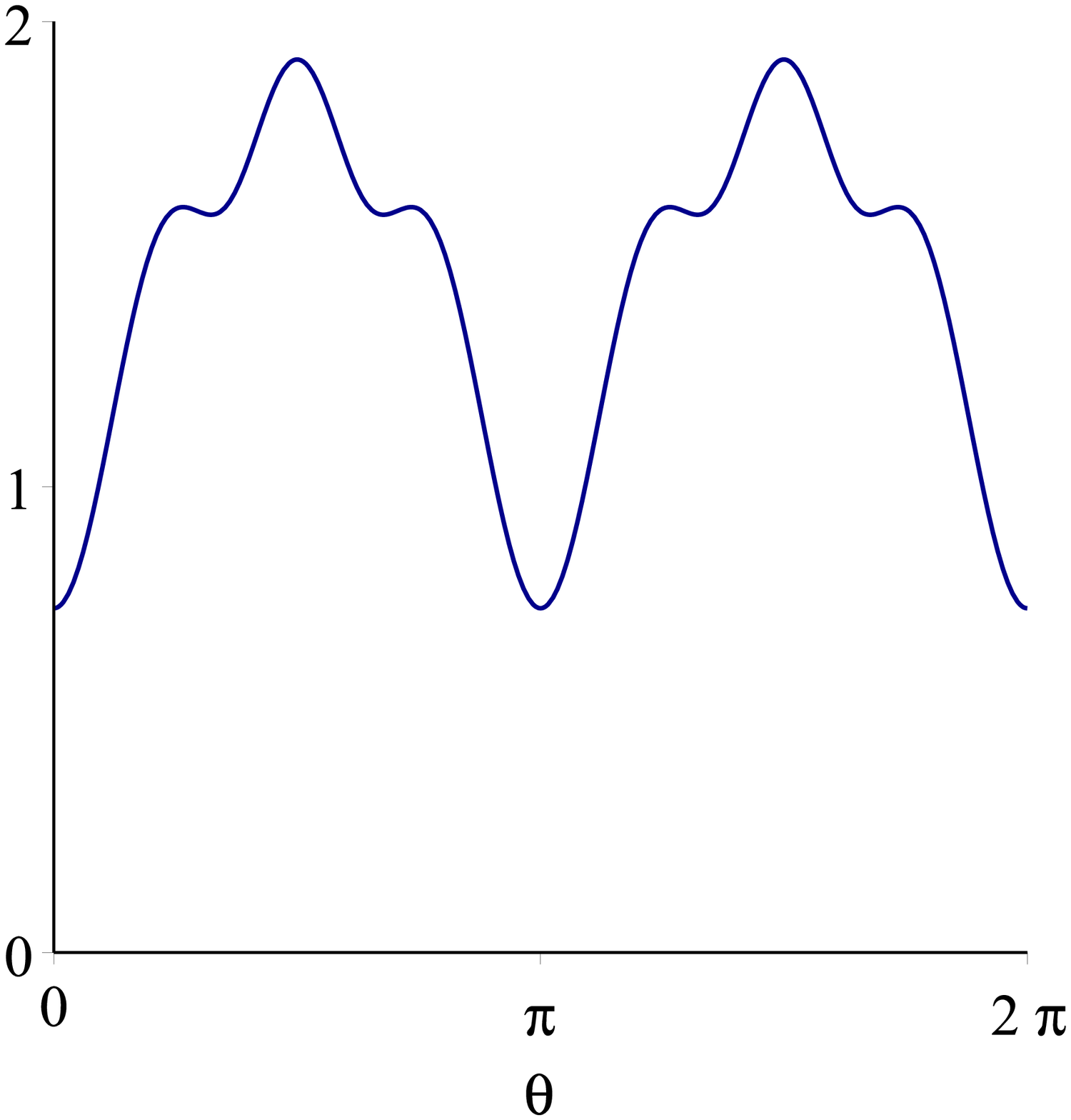}
\end{overpic}\\
\begin{overpic}
[width=4.8cm,height=5.0cm]{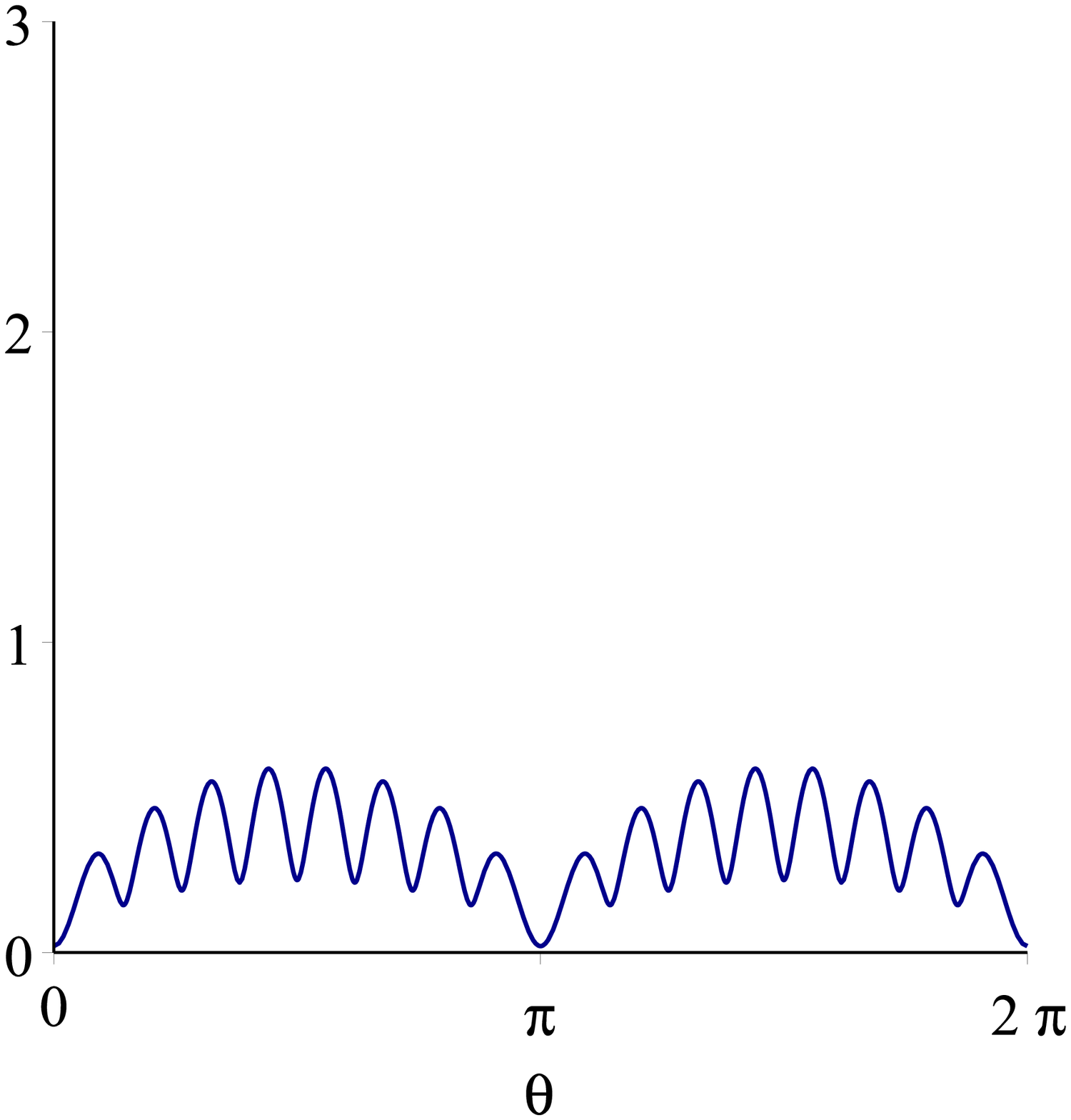}
\end{overpic}
\begin{overpic}
[width=4.8cm,height=5.0cm]{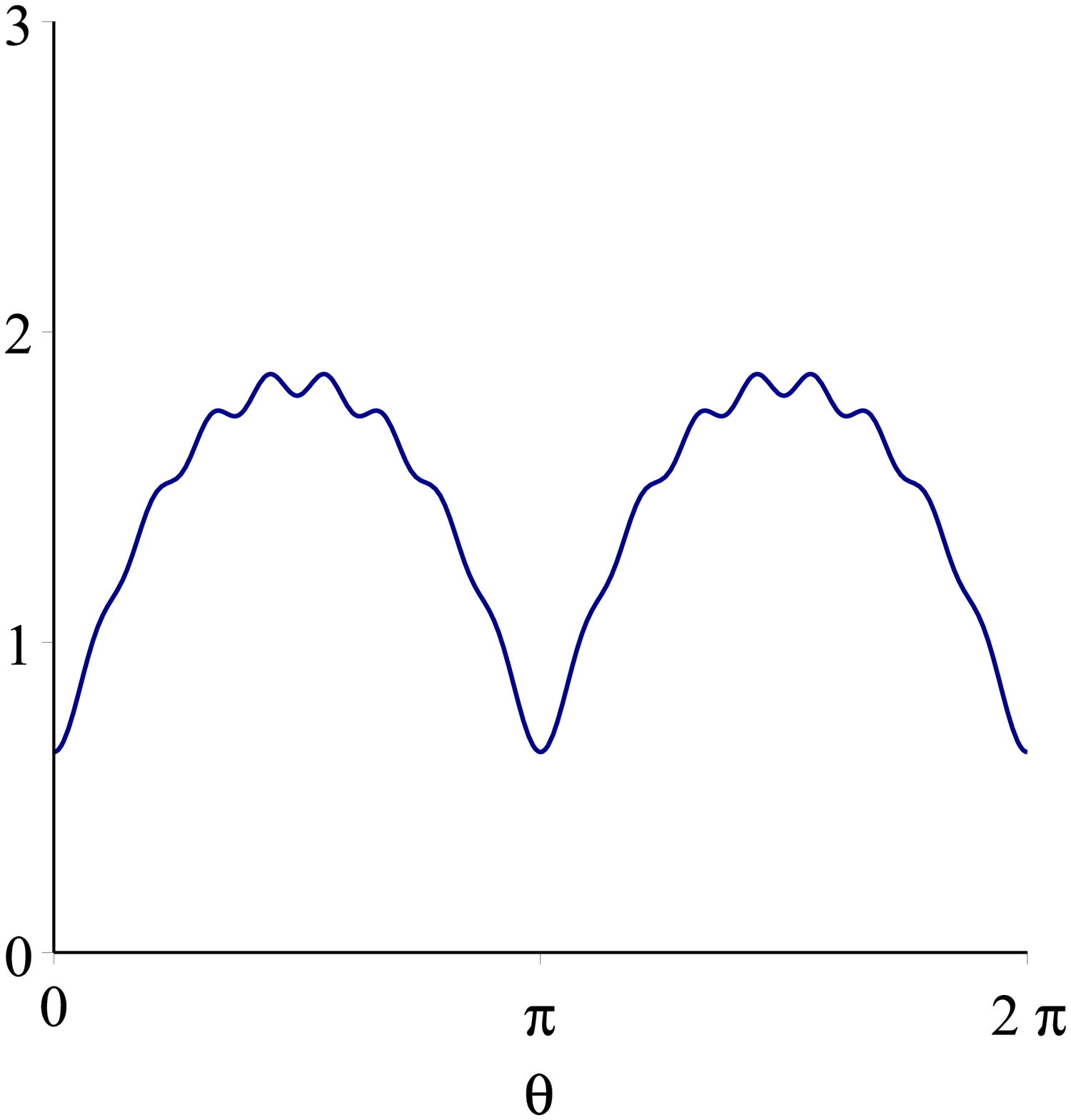}
\end{overpic}
\begin{overpic}
[width=4.8cm,height=5.0cm]{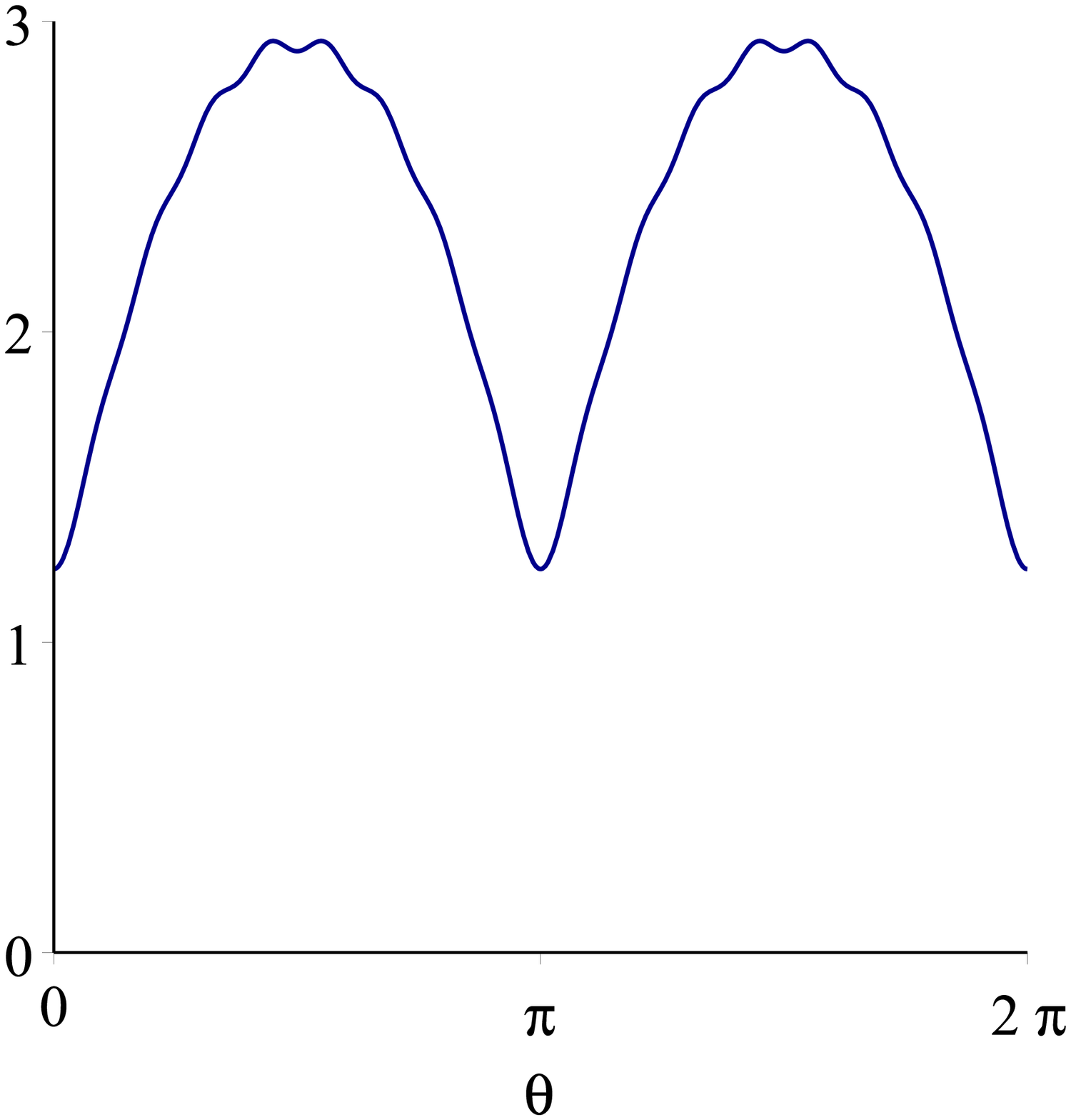}
\end{overpic}
\caption{The plot of $|\phi_n^{\mathrm{LGL}}(z)|$ as a function of
$\theta$. Top row shows $n=3$ and $\rho=1.05$ (left), $\rho=1.25$
(middle) and $\rho=1.5$ (right). Bottom row shows $n=8$ and
$\rho=1.05$ (left), $\rho=1.25$ (middle) and $\rho=1.32$ (right).}
\label{fig:Ellipse}
\end{figure}

To provide some insights into this conjecture, after some
calculations we obtain
\begin{align}
\left| \phi_1^{\mathrm{LGL}}(z) \right| &= \frac{3}{8} \left(\rho^2
+ \frac{1}{\rho^2} - 2\cos(2\theta) \right), \nonumber \\
\left| \phi_2^{\mathrm{LGL}}(z) \right| &= \frac{5}{16} \left[
\left( \left(\rho^2 + \frac{1}{\rho^2}\right)^2 - 4\cos^2(2\theta)
\right) \left(\rho^2 + \frac{1}{\rho^2} - 2\cos(2\theta) \right)
\right]^{1/2}. \nonumber
\end{align}
It is easily seen that the minimum values of $\left|
\phi_1^{\mathrm{LGL}}(z) \right|$ and $\left|
\phi_2^{\mathrm{LGL}}(z) \right|$ are always attained at
$\theta=0,\pi$, which confirms the above conjecture for $n=1,2$. For
$n\geq3$, however, $\left| \phi_n^{\mathrm{LGL}}(z)\right|$ will
involve a rather lengthy expression and it would be infeasible to
find the minimum value of $\left| \phi_n^{\mathrm{LGL}}(z) \right|$
from its explicit expression. We will pursue the proof of this
conjecture in future work.

\begin{example}
We consider the following Runge function
\begin{align}
f(x)= \frac{1}{1+(ax)^2}, \quad  a>0.
\end{align}
It is easily verified that this function has a pair of poles at
$z=\pm i/a$ and thus the rates of convergence of $p_n(x)$ is
$O(\rho^{-n})$ with $\rho=(1+\sqrt{a^2+1})/a$. In Figure
\ref{fig:LGLExam1} we illustrate the maximum errors of LGL
interpolants for two values of $a$. In our implementation, the LGL
interpolants $p_n(x)$ are computed by using the second barycentric
formula
\begin{align}\label{eq:bary}
p_n(x) = \frac{\displaystyle \sum_{j=0}^{n}\frac{w_j}{x-x_j}
f(x_j)}{\displaystyle \sum_{j=0}^{n}\frac{w_j}{x-x_j}},
\end{align}
where $\{w_j\}_{j=0}^{n}$ are the barycentric weights of LGL points
and the algorithm for computing these barycentric weights is
described in \cite{wang2014explicit}. Clearly, we see that numerical
results are in good agreement with our theoretical analysis.
\begin{figure}[h]
\centering
\begin{overpic}
[width=7cm,height=6.cm]{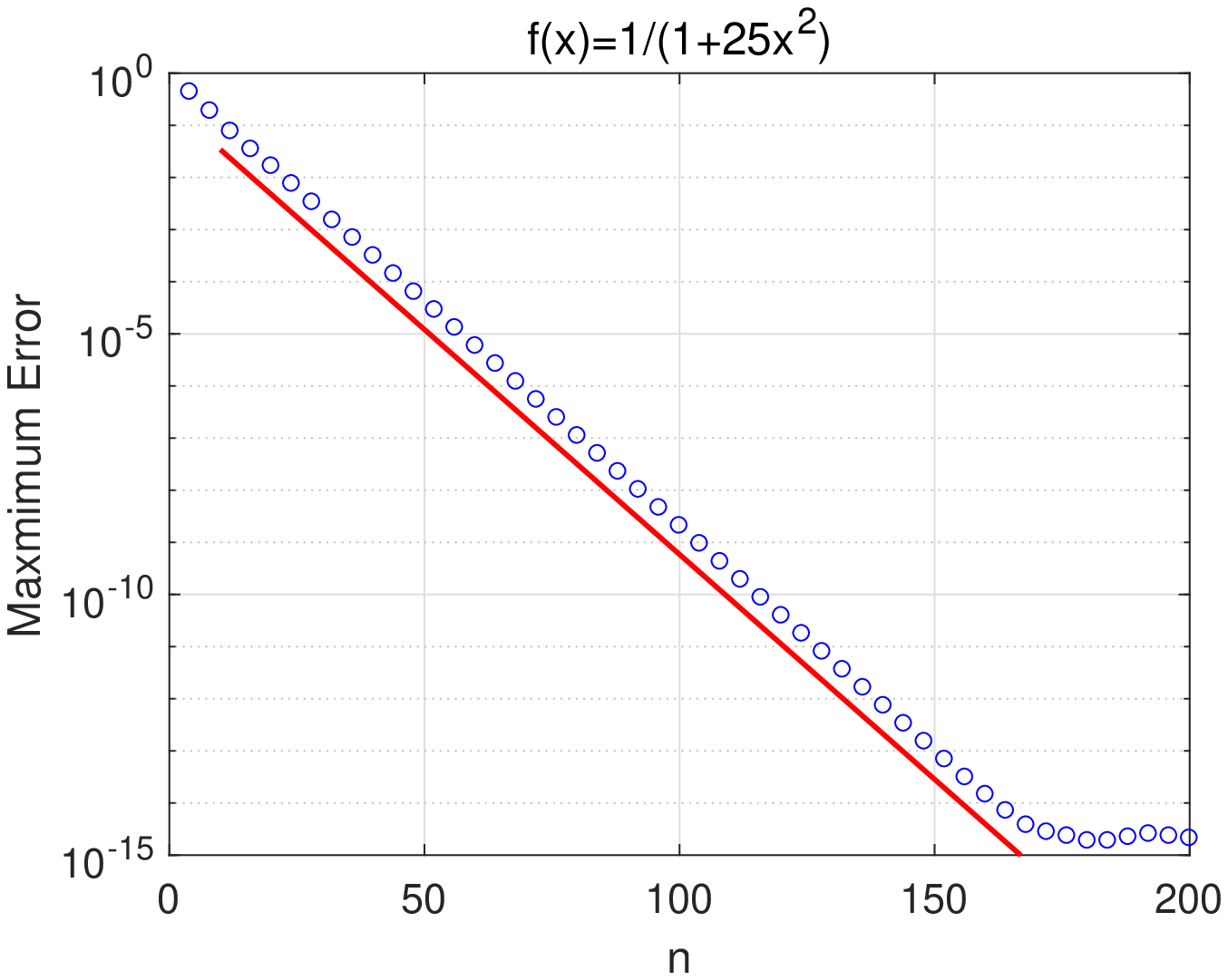}
\end{overpic}
\begin{overpic}
[width=7cm,height=6.cm]{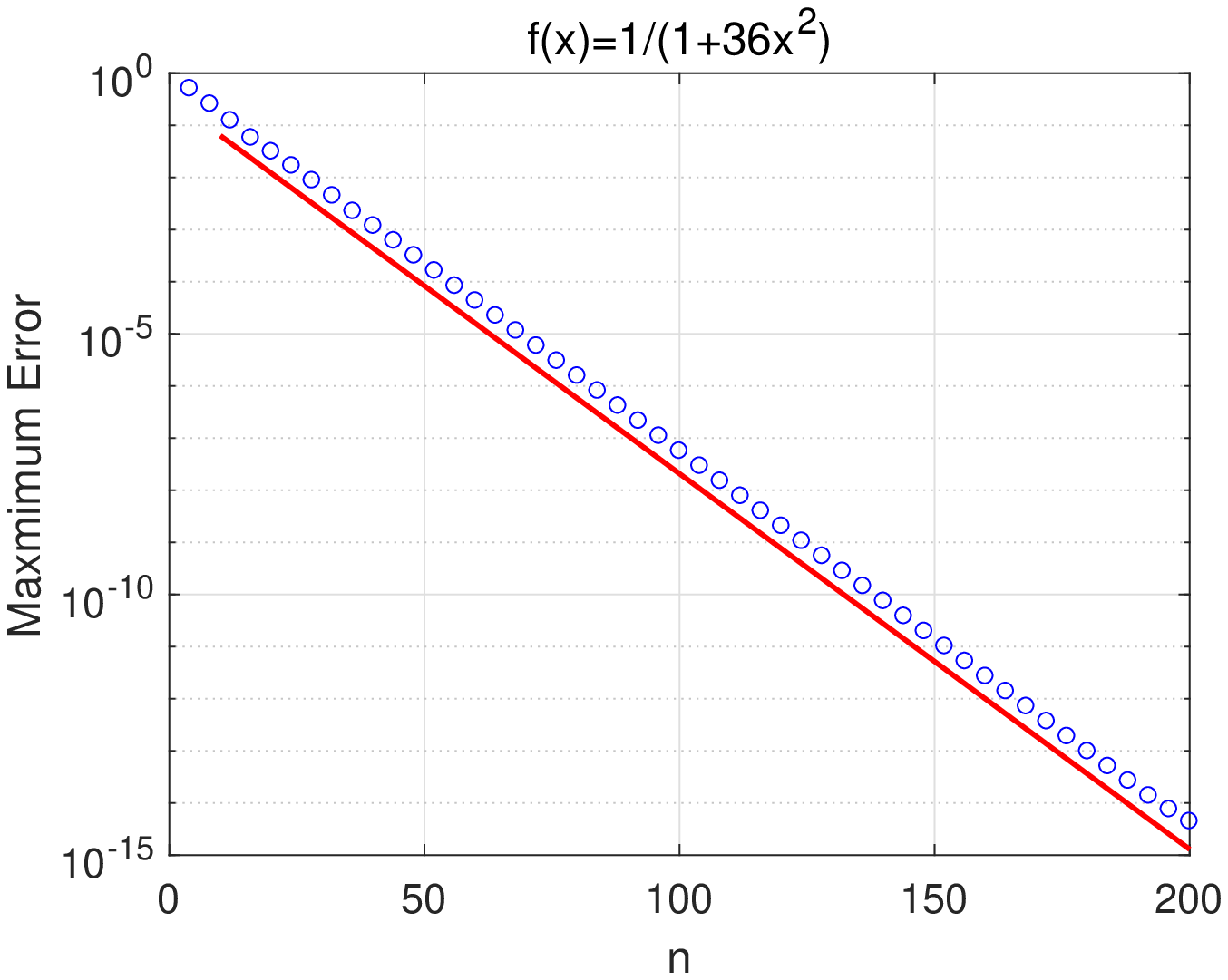}
\end{overpic}
\caption{The maximum error $\|f-p_n\|_{L^{\infty}(\Omega)}$
(circles) and the predicted convergence rate (line) for $a=5$ (left)
and $a=6$ (right).} \label{fig:LGLExam1}
\end{figure}
\end{example}

In Figure \ref{fig:LGLExam2} we illustrate the maximum errors of LGL
spectral differentiation. In our implementation, we first compute
the barycentric weights $\{w_j\}_{j=0}^{n}$ by the algorithm in
\cite{wang2014explicit} and then compute the differentiation matrix
$D$ by
\begin{align}
D_{j,k} = \left\{
\begin{array}{ll}
   {\displaystyle \frac{w_k/w_j}{x_j-x_k}}, & \hbox{$j\neq k$,}   \\[15pt]
   {\displaystyle -\sum_{k\neq j} D_{j,k}}, &\hbox{$j=k$.} % -\sum_{k=0,k\neq j}^{n} D_{j,k}^{(1)}
\end{array}
      \right.
\end{align}
Finally, $(p_n{'}(x_0),\ldots,p_n{'}(x_n))^T$ is evaluated by
multiplying the differentiation matrix $D$ with
$(f(x_0),\ldots,f(x_n))^T$. From \eqref{eq:LGLDiffB} we know that
the predicted rate of convergence of LGL spectral differentiation is
$O(n^{3/2}((1+\sqrt{a^2+1})/a)^{-n})$. Clearly, we can observe from
Figure \ref{fig:LGLExam2} that the predicted rate of convergence is
consistent with the errors of LGL spectral differentiation.

\begin{figure}[h]
\centering
\begin{overpic}
[width=7cm,height=6.cm]{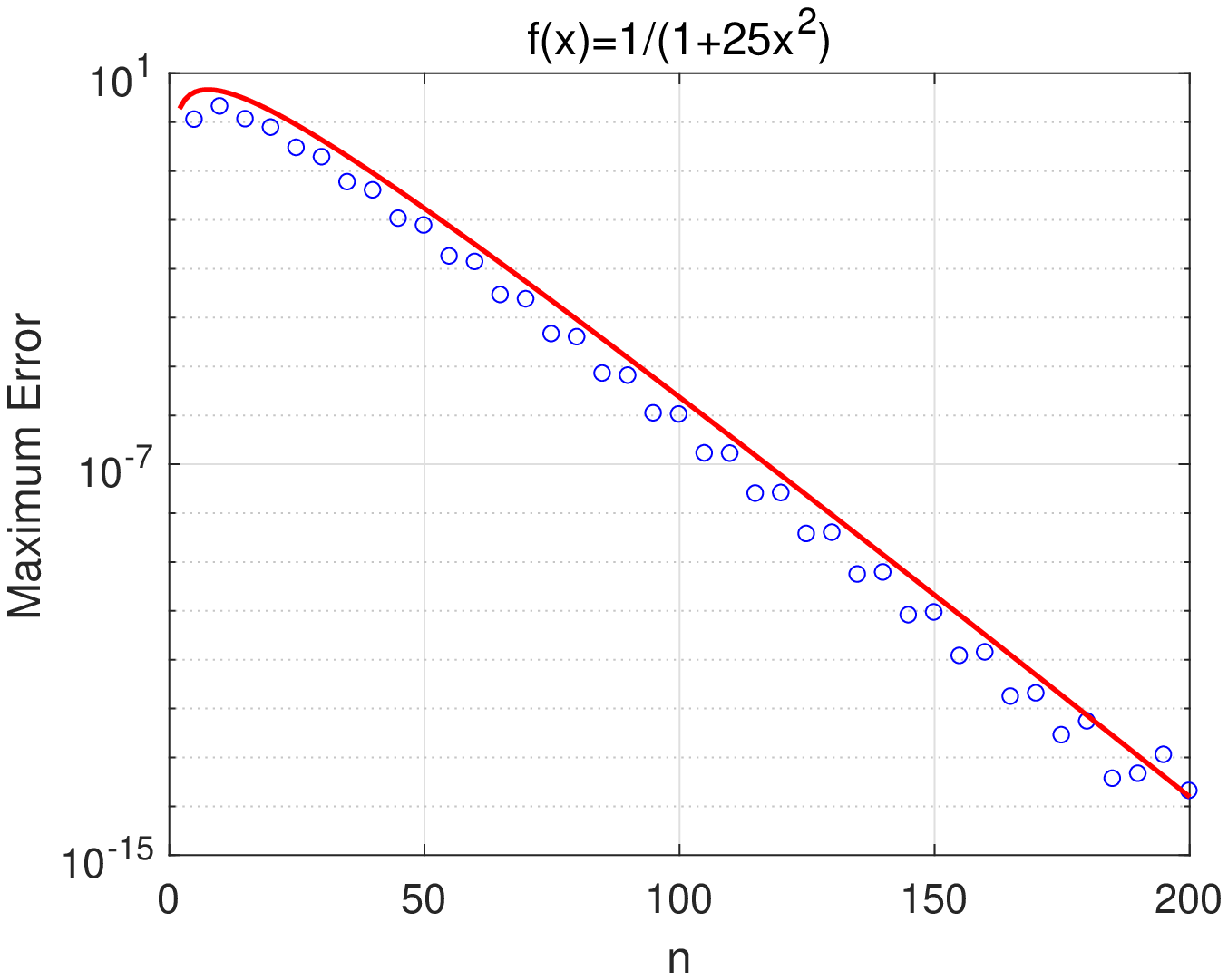}
\end{overpic}
\begin{overpic}
[width=7cm,height=6.cm]{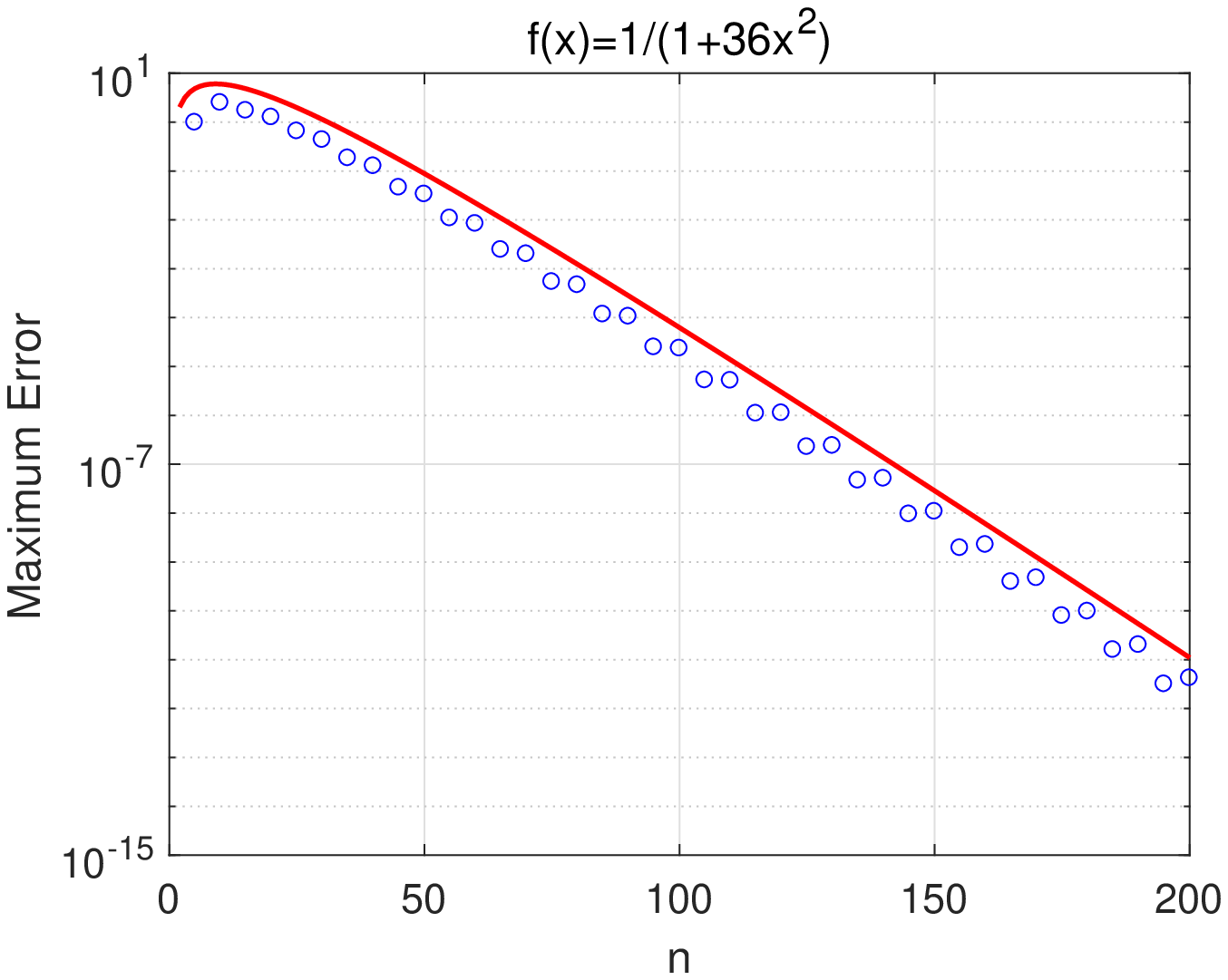}
\end{overpic}
\caption{The maximum error of LGL spectral differentiation (circles)
and the predicted convergence rate (line) for $a=5$ (left) and $a=6$
(right).} \label{fig:LGLExam2}
\end{figure}

\section{Conclusion}\label{sec:Conclusion}
In this paper, we have studied the error analysis of Legendre
approximations for differentiable functions from a new perspective.
We introduced the sequence of LGL polynomials
$\{\phi_0^{\mathrm{LGL}}(x),\phi_1^{\mathrm{LGL}}(x),\ldots\}$ and
proved their theoretical properties. Based on these properties, we
derived a new explicit bound for the Legendre coefficients of
differentiable functions. We then obtained an optimal error bound of
Legendre projections in the $L^2$ norm and an optimal error bound of
Legendre projections in the $L^{\infty}$ norms under the condition
that the maximum error of Legendre projections is attained in the
interior of $\Omega$. Numerical examples were provided to
demonstrate the sharpness of our new results. Finally, we presented
two extensions of our analysis, including Gegenbauer-Gauss-Lobatto
(GGL) functions
$\{\phi_1^{\mathrm{GGL}}(x),\phi_2^{\mathrm{GGL}}(x),\ldots\}$ and
optimal convergence rates of Legendre-Gauss-Lobatto interpolation
and differentiation for analytic functions.

In future work, we will explore the extensions of the current study
to a more general setting, such as finding some new upper bounds for
the weighted Jacobi polynomials (see, e.g., \cite{Krasikov2007}) and
establishing some sharp error bounds for Gegenbauer approximations
of differentiable functions.

\section*{Acknowledgements}
This work was supported by the National Natural Science Foundation
of China under grant 11671160. The author wishes to thank the editor
and two anonymous referees for their valuable comments on the
manuscript.

%\section*{Declarations}
%{\bf Conflict of interest} The author declares no competing
%interest.


\begin{thebibliography}{10}
\bibitem{Antonov1981}
V. A. Antonov and K. V. Hol\v{s}evnikov, An estimate of the
remainder in the expansion of the generating function for the
Legendre polynomials (Generalization and improvement of Bernstein's
inequality), Vestnik Leningrad Univ. Math., 13:163--166, 1981.

\bibitem{babuska2019}
I. Babu\v{s}ka and H. Hakula, Pointwise error estimate of the
Legendre expansion: The known and unknown features, Comput. Methods
Appl. Mech. Engrg., 345(1):748--773, 2019.

\bibitem{canuto2006spectral}
C. Canuto, M. Y. Hussaini, A. Quarteroni and T. A. Zang, Spectral
Methods: Fundamentals in Single Domains, Springer, 2006.

\bibitem{davis1975interp}
P. J. Davis, Interpolation and Approximation, Dover Publications,
New York, 1975.

\bibitem{durand1975}
L. Durand, Nicholson-type integrals for products of Gegenbauer
functions and related topics, Theory and Application of Special
Functions, Edited by Richard A. Askey, pp.353--374, Academic Press,
New York, 1975.

\bibitem{ern2021}
A. Ern and J.-L. Guermond, Finite elements I: Approximation and
Interpolation, Vol. 72 of Texts in Applied Mathematics, Springer,
Cham, 2021.

\bibitem{gautschi2004Orth}
W. Gautschi, Orthogonal Polynomials: Computation and Approximation,
Oxford University Press, London, 2004.

\bibitem{Jackson1930}
D. Jackson, The Theory of Approximation, American Mathematical
Society Colloquium Publications, Volume XI, New York, 1930.

\bibitem{Krasikov2007}
I. Krasikov, An upper bound on Jacobi polynomials, J. Approx.
Theory, 149:116--130, 2007.

\bibitem{liu2020legendre}
W.-J. Liu, L.-L. Wang and B.-Y. Wu, Optimal error estimates for
Legendre expansions of singular functions with fractional
derivatives of bounded variation, Adv. Comput. Math., 47: article
number: 79, 2021.

\bibitem{Mason2003}
J. C. Mason and D. C. Handscomb, Chebyshev Polynomials, Chapman and
Hall/CRC, Boca Raton, 2003.

\bibitem{olver2010nist}
F. W. J. Olver, D. W. Lozier, R. F. Boisvert and C. W. Clark, NIST
Handbook of Mathematical Functions, Cambridge University Press,
2010.

\bibitem{protter1991real}
M. H. Protter and C. B. Morrey, A First Course in Real Analysis,
Second Edition, Springer-Verlag, New York, 1991.

\bibitem{shen2011spectral}
J. Shen, T. Tang and L.-L. Wang, Spectral Methods: Algorithms,
Analysis and Applications, Springer, Heidelberg, 2011.

\bibitem{szego1975orth}
G. Szeg\H{o}, Orthogonal Polynomials, Vol.~23, 4th Edition, Amer.
Math. Soc., Providence, RI, 1975.

\bibitem{trefethen2019}
L. N. Trefethen, Approximation Theory and Approximation Practice,
Extended Edition, SIAM, Philadephia, 2019.

\bibitem{wang2012legendre}
H.-Y. Wang and S.-H. Xiang, On the convergence rates of Legendre
approximation, Math. Comp., 81(278):861--877, 2012.

\bibitem{wang2014explicit}
H.-Y. Wang, D. Huybrechs and S. Vandewalle, Explicit barycentric
weights for polynomial interpolation in the roots or extrema of
classical orthogonal polynomials, Math. Comp., 83(290):2893--2914,
2014.

\bibitem{wang2016gegenbauer}
H.-Y. Wang, On the optimal estimates and comparison of Gegenbauer
expansion coefficients, SIAM J. Numer. Aanl., 54(3):1557--1581,
2016.

\bibitem{wang2018new}
H.-Y. Wang, A new and sharper bound for Legendre expansion of
differentiable functions, Appl. Math. Lett., 85:95--102, 2018.

\bibitem{wang2018jacobi}
H.-Y. Wang and L. Zhang, Jacobi polynomials on the Bernstein
ellipse, J. Sci. Comput., 75:457--477, 2018.

\bibitem{wang2021legendre}
H.-Y. Wang, How much faster does the best polynomial approximation
converge than Legendre projections?, Numer. Math., 147:481--503,
2021.

\bibitem{wang2021gegenbauer}
H.-Y. Wang, Optimal rates of convergence and error localization of
Gegenbauer projections, IMA J. Numer. Anal.,
\texttt{https://doi.org/10.1093/imanum/drac047}, 2022.

\bibitem{wang2021cheby}
H.-Y. Wang, Analysis of error localization of Chebyshev spectral
approximations, SIAM J. Numer. Anal., 61(2):952--972, 2023.

\bibitem{xiang2020jacobi}
S.-H. Xiang and G.-D. Liu, Optimal decay rates on the asymptotics of
orthogonal polynomial expansions for functions of limited
regularities, Numer. Math., 145:117--148, 2020.

\bibitem{xie2013exp}
Z.-Q. Xie, L.-L. Wang and X.-D. Zhao, On exponential convergence of
Gegenbauer interpolation and spectral differentiation, Math. Comp.,
82(282):1017--1036, 2013.

\end{thebibliography}
\end{document}